\documentclass[letterpaper,10pt]{article}
\usepackage{amssymb,amsmath,amsthm,turnstile,txfonts}
\usepackage{color}
\usepackage{mathrsfs}
\usepackage{graphicx}
\usepackage{url}
\usepackage{multirow}
\usepackage[left=1in,top=1in,right=1in,bottom=1in,letterpaper]{geometry}
\usepackage{setspace}
\usepackage[algo2e,linesnumbered, vlined,ruled]{algorithm2e}
\usepackage{wrapfig}
\onehalfspace

\def\M{\mathcal{M}}

\def\N{\mathcal{N}}

\def\P{\mathbb{X}}
\def\Q{\mathbb{Y}}

\def\sP{\mathscr{P}}
\def\sB{\mathscr{B}}
\def\P{\mathbb{P}}
\def\Q{\mathbb{Q}}

\def\T{\mathcal{T}}

\def\d{\mathrm{d}}

\def\RSWD{\mathrm{RSWD}}
\def\RW{\mathrm{RW}}
\def\RWD{\mathrm{RWD}}

\def\Tr{\mathrm{Tr}}

\def\RR{\mathbb{R}}
\def\ADM{\mathrm{ADM}}

\newtheorem{theorem}{{\bf Theorem}}
\newtheorem{proposition}{{\bf Proposition}}

\newtheorem{remark}{\bf Remark}

\newtheorem{corollary}{\bf Corollary}

\begin{document}
\title{Multi-scale Non-Rigid Point Cloud Registration Using Robust Sliced-Wasserstein Distance via Laplace-Beltrami Eigenmap}
\author{Rongjie Lai,  Hongkai Zhao  \\
Department of mathematics, UC Irvine \\
   rongjiel, zhao@math.uci.edu
  \thanks{The research of Hongkai Zhao is partially supported by ONR grant N00014-11-1-0602 and NSF grant DMS-1115698. The research of Rongjie Lai is partially supported by AMS-Simons travel grants.}
}

\date{}
\maketitle

\begin{abstract}
In this work, we propose computational models and algorithms for point cloud registration with non-rigid transformation. First,  point clouds sampled from manifolds originally embedded in some Euclidean space $\RR^D$ are transformed to new point clouds embedded in $\RR^n$ by Laplace-Beltrami(LB) eigenmap using the $n$ leading eigenvalues and corresponding eigenfunctions of LB operator defined intrinsically on the manifolds. The LB eigenmap are invariant under isometric transformation of the original manifolds.  
Then we design computational models and algorithms for registration of the transformed point clouds  in distribution/probability form based on the optimal transport theory which provides both generality and flexibility to handle general point clouds setting.  Our methods use robust sliced-Wasserstein distance, which is as the average of projected Wasserstein distance along different directions, and incorporate a rigid transformation to handle ambiguities introduced by the Laplace-Beltrami eigenmap. By going from smaller $n$, which provides a quick and robust registration (based on coarse scale features) as well as a good initial guess for finer scale registration, to a larger $n$, our method provides an efficient, robust and accurate approach for multi-scale non-rigid point cloud registration. 
\end{abstract}

\section{Introduction}


Geometric modeling and analysis of surfaces and manifolds,  such as  shape comparison, classification, registration and manifold learning, 
has important applications in many fields including computer vision, computer graphics, medical imaging and data analysis \cite{Cootes:1995,Funkhouser:03,Thompson:03,Gu:2004,lin2008riemannian}. Unlike images, which typically have a canonical form of representation as functions defined on a uniform grid in a rectangular domain, surfaces and manifolds in 3D and higher are with more complicated geometrical structures and do not have a canonical or natural form of representation or global parametrization.                     
Moreover, they are typically represented as embedding manifolds in $\mathbb{R}^D$, where it is well known that intrinsically identical manifolds can have significantly different representations due to actions such as translation, rotation, and more general isometric transformations.  These intrinsic mathematical difficulties make geometric modeling and analysis in practice much more challenging than processing 2D images. One important task in practice is to construct distinctive and robust intrinsic features or representations, local and/or global, that can be used to characterize, analyze, compare and classify shapes, surfaces and manifolds in general. 

There have been a lot of study and proposed methods for shape analysis based on intrinsic geometry.  In early works, feature-based methods were  developed in computer vision and graphics to compare surfaces in an intrinsic fashion. Various features were proposed such as shape contexts, shape distributions, shape inner distance, conformal factor ~\cite{Osada:02,Belongie:02,Elad:2003,Ling:07,Tangelder:MTP2008,Ben-Chen:PEWSR2008} to characterize various aspects of surface geometry. These features are usually application specific and do not render a  metric structure in shape space. The shape-space approach overcomes this difficulty by introducing metric structures on the space of all surfaces, where the distance between two surfaces can be measured by the metric structure introduced for the shape space. For instance, the Teichm$\ddot{u}$ller 
space, geodesic spectra and the computation of Teichm$\ddot{u}$ller coordinates is discussed in~\cite{Luo:JDG1998,Jin:2007ACMSSPM,Jin:2009}. However, important local features are lost in this approach because each surface is viewed as 
a point in the shape space. More recently, another class of approaches were introduced based on the metric geometry~\cite{Burago:GSM2001}. In this approach, each surface is treated as a metric space and surfaces are compared according to the theory of metric geometry by measuring their Gromov-Hausdorff distance and Gromov-Wasserstein distance~\cite{Memoli:2005FCM,Memoli:2007SPBG, Memoli:2008CVPR,Memoli:2011ACM}. Diffusion distance~\cite{Lafon:Phdthesis,Coifman:2006ACHA} on surface was applied to compute the Gromov-Hausdorff distance  and compare non-rigid geometry based on the metric geometry \cite{Bronstein:2009IJCV}. While theoretically this class of methods can compute both local and global surface differences, the need for optimization over all possible correspondences makes it computationally challenging to conduct detailed analysis of surface structures in practice.

More recently, there have been increasing interests in using the Laplace-Beltrami (LB) eigensystem for 3D shape and surface analysis. Mathematically, the LB eigensystem provides an intrinsic and systematic characterization of the underlying geometry. Practically, LB eigensystem is very computable by well developed numerical methods. 
 A series of interesting works have been developed for surface characterization and analysis using its LB eigensystem~\cite{Reuter:06,Levy:2006IEEECSMA,Vallet:2008CGF,shen:2006SIGGRAPH,Shi:08a,Reuter:ijcv09,Shi:TMI2010,Peter:08,Sun:2009SGP, Bronstein:2010CVPR}. 
For surface registration, the LB eigen-system has been used to construct a common embedding space to measure shape differences, such as heat kernel embedding~\cite{Berard:1994} and global position signature~\cite{Rustamov:2007}. After embeddings, nonrigid (near) isometric deformations can be handled naturally due to the intrinsic properties of LB eigensystem.  However, there are a few new challenges for the embedding using LB eigensystem. First, there are ambiguities in LB eigensystem just like any eigensystem such as sign ambiguity for eigenfunctions and ambiguity due to the multiplicity of eigenvalues. Second, numerically, two eigenfunctions may switch order if their corresponding eigenvalues are too close. Both ambiguities can be modeled as an unknown rigid transformation. Third, the dimension of LB embedding space is usually high in order to capture fine features. 




In this paper, we propose a multi-scale nonrigid point cloud registration based on LB eigenmap.  In practice shapes, surfaces or manifolds in general can be sampled/represented in the simplest and most basic form, point clouds, in any space dimensions. For examples, these point clouds could 
come from laser scanners or vertices of triangulated surfaces in 3D modeling, or high dimensional feature vectors in machine learning.
We first transform the original point clouds to new point clouds in $\RR^n$ by LB eigenmap, which is defined in Section \ref{sec:ManifoldRegistration}, using the $n$ leading eigenvalues and corresponding eigenfunctions for LB operator defined intrinsically on the manifolds. In particular LB eigenmap can remove isometric variance in the original point clouds. 
Hence the original point cloud registration problem becomes a registration problem for the transformed point clouds. We use optimal transportation theory to model the registration or mapping between point clouds in a distribution sense with a natural probability interpretation. This formulation has both generality and flexibility to deal with point clouds to incorporate uncertainty, prior knowledge and other information.
 By incorporating an unknown rigid transformation in the non-convex optimization problem, one can overcome possible ambiguities of LB eigensystem by optimizing the {\em robust Wasserstein distance (RWD)} defined in section~\ref{sec:ManifoldRegistration}. To overcome the computation complexity of RWD distance, we further propose a new distance, called {\em robust sliced-Wasserstein distance (RSWD)}, defined in the new embedding space by LB eigenmap. Using RSWD, which is the average of projected Wasserstein distance along different directions, significantly reduces the computation cost. Moreover, we show that both RWD and RSW provide a metric in the LB embedding space which naturally quotient out all isometric transformations in the original embedding space. Meanwhile, we design efficient algorithms to solve the optimization problem. Monotonicity of the algorithm can be guaranteed although the proposed optimization problem is non-convex. Finally, By going from smaller $n$, which provides a quick and robust registration based on large scale features as well as a good initial guess for finer scale registration, to a larger $n$, our method provides an efficient, robust and accurate  multi-scale non-rigid point cloud or manifold registration.


The rest of the paper is organized as follows. In Section~\ref{sec:mathbackground}, we briefly review the mathematical background of LB eigensytem defined on Riemannian manifolds and discuss a few existing numerical methods to solve the LB eigenvalue problem on point clouds.
After that,
in Section~\ref{sec:ManifoldRegistration}, we model the registration problem as a non-convex optimization problem based on RWD in the new embedding space defined by LB eigenmap which can handle the ambiguities inherited in the LB eigensystem.
Motivated by the optimization problem in Section~\ref{sec:ManifoldRegistration}, we propose RSWD distance in Section~\ref{sec:RSWdistance}, this new distance has advantages for being both theoretically sound and computationally efficient. In the end, a multi-scale registration process which can further improve both efficiency and robustness is developed based on the multi-scale nature of the LB eigenfunctions. 
To demonstrate the effectiveness of the proposed models and algorithms, several computational examples are illustrated in Section~\ref{sec:NumericalResults}. Finally, conclusions are made in Section~\ref{sec:conclusion}.

\section{Laplace-Beltrami Eigensystem on Riemannian Manifolds}
\label{sec:mathbackground}
In this section, we first briefly present mathematical background of Laplace-Beltrami (LB) eigen-geometry. After that, existing numerical methods of computing LB eigensystem on triangulated surfaces or point clouds will be discussed. 

\subsection{Mathematical background}
\label{subsec:LBeigs}
Let $(\M,g)$ be a closed $d-$dimensional Riemannian manifold. For any smooth function
$\phi \in C^{\infty}(\M)$, the LB operator in a local coordinate system
$\{(x^1,x^2,\cdots,x^d)\}$ is defined as the following coordinate invariant form~\cite{Chavel:1984,Schoen:1994,Jost:2001}: 
\begin{eqnarray} \Delta_{(\M,g)}
\phi=\frac{1}{\sqrt{G}}\sum_{i=1}^{d}\frac{\partial}{\partial
x_i}(\sqrt{G}\sum_{j=1}^{d}g^{ij}\frac{\partial \phi}{\partial x_j})
\end{eqnarray}  
where $(g^{ij})$ is the inverse matrix of $g = (g_{ij})$ and
$G=\det(g_{ij})$. In the rest of the paper, we are interested in $\M$ with dimension $d \geq 2$.

The LB operator is self-adjoint and elliptic, so its
spectrum is discrete. We denote the eigenvalues of $\Delta_{\M}$
as $0=\lambda_0<\lambda_1<\lambda_2<\cdots$ and the corresponding
eigenfunctions as $\phi_0, \phi_1,\phi_2,\cdots$ such that 
\begin{eqnarray} 
\Delta_{(\M,g)} \phi_k=-\lambda_k\phi_k, \quad \& \quad  \int_{\M} \phi_k(x) \phi_k(x) \sqrt{G} \d x = 1, \quad k=0,1,2,\cdots.
\label{eqn:lb_closesurf}
\end{eqnarray}  
The set of LB eigenfunctions $\{\phi_n\}_{n=0}^\infty$ forms an orthonormal basis of the space of $L^2$ functions on $\M$. The set 
$\{\lambda_n,\phi_n\}_{n=0}^\infty$ is called an LB eigen-system of $(\M,g)$.

Due to the intrinsic definition of LB operator $\Delta_{\M}$ , the induced LB eigen-system
$\{\lambda_i, \phi_i\}_{i=0}^ \infty$ is also completely intrinsic to the manifold geometry. 
In particularly, the LB eigen-system is invariant under isometric transformations that are rigid or nonrigid. In practice, scale invariant is also desirable. The following scaling property of LB eigen-system leads to a scale invariant embedding which will be introduced in the next section. 
\begin{proposition}
\label{prop:scaleformula}
Let $c$ be a positive constant, $\{\lambda_i,\phi_i\}$ be the eigensystem of
$(\M,g)$, and $\{\tilde{\lambda}_i,\tilde{\phi}_i\}$ be the eigensystem
of $(\M,cg)$, then
\begin{eqnarray} \lambda_i=c\cdot\tilde{\lambda_i} ~~\mbox{and}~~
\phi_i=c^{d/4}\cdot\tilde{\phi}_i \label{formula:scale}
\end{eqnarray}  
\end{proposition}
\begin{proof} $\lambda_i=c\cdot\tilde{\lambda_i} $ can be obtained by 
\begin{equation}
\Delta_{(\M,cg)}
\phi=\frac{1}{\sqrt{c^d G}}\sum_{i=1}^{d}\frac{\partial}{\partial
x_i}(\sqrt{c^d G}\sum_{j=1}^{d}c^{-1} g^{ij}\frac{\partial \phi}{\partial x_j}) = \frac{1}{c}\Delta_{(\M,g)}
\end{equation}
Due to the normalization condition, we have
\begin{equation}
1=  \int_{\M} \phi_i(x) \phi_i(x) \sqrt{G} \d x = \int_{\M} \phi_i(x) \phi_i(x) c^{-d/2} \sqrt{c^d G} \d x = \int_{\M} \phi_i(x) c^{-d/4}\phi_i(x) c^{-d/4} \sqrt{c^d G} \d x
\end{equation}
This yields $\phi_i=c^{d/4}\cdot\tilde{\phi}_i $.
\end{proof}

The eigenfunctions of the LB
operator are natural extensions of the Fourier basis functions from a
Euclidean domain to a general manifold. A special example is the
spherical harmonics, which are the eigenfunctions of the LB operator
on the unit sphere, and proved to be useful for representation and study functions
defined on spheres.  In addition to being a useful tool for function analysis on manifolds, 
the LB eigen-system also provides an important tool for characterization and study of the intrinsic geometry of the underlying manifold.
One of the most famous examples to study surface geometry using LB
eigen-systems is Kac's question~\cite{Kac:1966}:~\textquotedblleft Can one 
hear the shape of a drum?\textquotedblright~Namely, can we determine the 
geometry of surfaces using their LB eigenvalues? However, LB eigenvalues, which are 
just a part of the LB eigen-system, can not completely determine surface geometry
\cite{Milnor:1964,Sunada:1985,Protter:1987,Gordon:1992}. More complete information is stored
in LB eigenfunctions. Uhlenbeck~\cite{Uhlenbeck:76} proved that LB eigenfunctions are generically Morse functions which provides theoretical evidence that global information of a manifold can be obtained by its LB eigenfunctions. B\'{e}rard et. al.\cite{Berard:1994} 
introduced the first theoretical result about using LB eigen-system as global 
embedding to analyze manifolds, which shows that two manifolds have the same LB eigensystem if and only if they are isometric to each other. This mathematical result suggests that the LB eigen-system provides intrinsic signatures which can determine a manifold uniquely. 
In the past few years, there have been increasing 
interests in using LB eigen-systems for 3D shape and surface analysis. 
M. Reuter~\cite{Reuter:06} proposed to use LB eigenvalues as fingerprints
to classify surfaces. Rustamov~\cite{Rustamov:2007} was one of the first to use global 
embedding obtained by LB eigen-system to study surfaces. P. Jones et al.
\cite{Peter:08} introduced a new manifold local parametrization approach 
using LB eigenfunctions. J. Sun, M. Bronstein et al.
\cite{Sun:2009SGP, Bronstein:2010CVPR} introduced intrinsic surface descriptors 
using heat kernel of surface heat equation. Several applications
of LB eigen-system in medical image analysis have been discussed in~\cite{Qiu:06,Shi:08a, Shi:08b, Lai_ISBI:09,Lai:2010CVPR}.
More recently, LB eigen-system is proposed for global understanding of point clouds data in \cite{Liang:CVPR2012}, which can be extended to study manifolds represented by point clouds in higher dimensions.

\subsection{Numerical methods for solving LB eigen-system on point clouds}
\label{subsec:computeLB}
There are several ways for solving the LB eigenvalue problem \eqref{eqn:lb_closesurf} for a given surface or a manifold sampled by point clouds. The corresponding numerical methods are designed for different representations of the given surface or manifold. If a global triangulation of a point cloud is available, differential operators on the triangulated manifold can be readily approximated \cite{Taubin:2000EUROGRAPHICS,Meyer:2003,Xu:2004}. Typically finite element method is a natural choice which turns the LB eigenvalue problem to the eigenvalue problem of a symmetric positive definite matrix \cite{Reuter:06,Qiu:06,Shi:08a}.  Another possible method, called closest point method \cite{Macdonald:2011}, is to lay down a uniform grid and use closest point relation to extend differential operators for functions defined on the surface to differential operators in the ambient Euclidean space. Finite difference method is typically used for the discretization. The resulting linear system can not be guaranteed symmetric or positive definite in general. 
There are also a few methods for solving LB eigenvalue problem directly on point clouds without a global mesh or grid. These methods are particular useful in high dimensions where a global triangulation or grid is intractable. There are two types of methods in this class. One type is kernel based methods, where the LB operator on point clouds is approximated by heat diffusion in the ambient Euclidean space~\cite{Coifman:2006ACHA} or in the tangent space~\cite{Belkin:09clp} among nearby points. In another word, the metric on the manifold is approximated by Euclidean metric locally.  This type of method does not need to approximate differential operators directly. The main advantage of such methods, e.g., diffusion map method \cite{Coifman:2006ACHA},  is their simplicity and generality. It can be applied to point clouds even not necessarily embedded in a metric space, e.g., graph Laplacian. The resulting linear system is an H-matrix, i.e., satisfying discrete maximum principle, but not symmetric.
However, these kernel based methods render low order approximation and is only limited to the approximation of the Laplace-Beltrami operator on point clouds. More recently, two new methods for
solving PDEs on point clouds are introduced in~\cite{Liang:CVPR2012,Liang2013solving,Lai:2013IPI}, where systematic methods are introduced to approximate differential operators intrinsically at each point from local approximation of the manifold and its metric by moving least squares through nearby neighbors. The resulting linear system for LB eigenvalue problem is also H-matrix but not symmetric. These methods can achieve high order accuracy and be used to approximate general differential operators on point clouds sampling manifolds with arbitrary dimensions and co-dimensions.  Moreover, the computational complexity depends mainly on the true dimension of the manifold rather than the dimension of the embedded space. However, the computational cost for local reconstruction of the manifold and its metric based on moving least squares grows quickly with the dimension of the manifold.

With these available numerical methods, the LB eigenvalues and LB eigenfunctions for a given Riemannian manifold can be obtained. In this paper, we mainly focus on the registration problem using solution from computed LB eigenvalue problem.

\section{A framework of nonrigid Manifolds Registration through LB eigenmap}
\label{sec:ManifoldRegistration}
Let $(\M, g_{\M})$ and $(\N, g_{\N})$ be two isometric d-dimensional Riemannian manifolds with induced metric from $\mathbb{R}^D$. A nonrigid isometric registration between $\M$ and $\N$ is a smooth diffeomorphism $f:\M \rightarrow \N$, such that the pullback metric satisfies $f^{*}(g_{\N}) = g_{\M}$. Due to the nonlinearity of $f$, it is not straightforward to compute $f$ in the ambient space $\RR^D$. Here, we consider to approximate $f$ through a registration between images of LB eigenmaps of $\M$ and $\N$. More specifically, 
let's write $\Phi_n = \{\lambda_k, \phi_k\}_{k=1}^{n}$ and $\Psi_n = \{\eta_k,\psi_k \}_{k=1}^{n}$ as the first $n$ nontrivial LB eigenvalues and LB eigenfunctions of $\M$ and $\N$ respectively, which can be computed using the numerical methods discussed in Section \ref{subsec:computeLB}. We consider the following LB eigenmap of $\M$ and $\N$ using their first $n$ nontrivial LB eigenvalues and LB eigenfunctions.
\begin{eqnarray}
\iota_{\M}^n:\M\rightarrow \mathbb{R}^n, & \displaystyle \quad u\mapsto p = \iota_{\M}^n(u) = \left(\frac{\phi_1(u)}{\lambda_1^{d/4}},\cdots, \frac{\phi_n(u)}{\lambda_n^{d/4}}\right) \nonumber\\
\iota_{\N}^n:\N\rightarrow \mathbb{R}^n, & \displaystyle  \quad v\mapsto q = \iota_{\N}^n(v) = \left(\frac{\psi_1(v)}{\eta_1^{d/4}},\cdots, \frac{\psi_n(v)}{\eta_n^{d/4}}\right)
\end{eqnarray}

\begin{figure}[h]
\begin{center}
\includegraphics[width=1\linewidth]{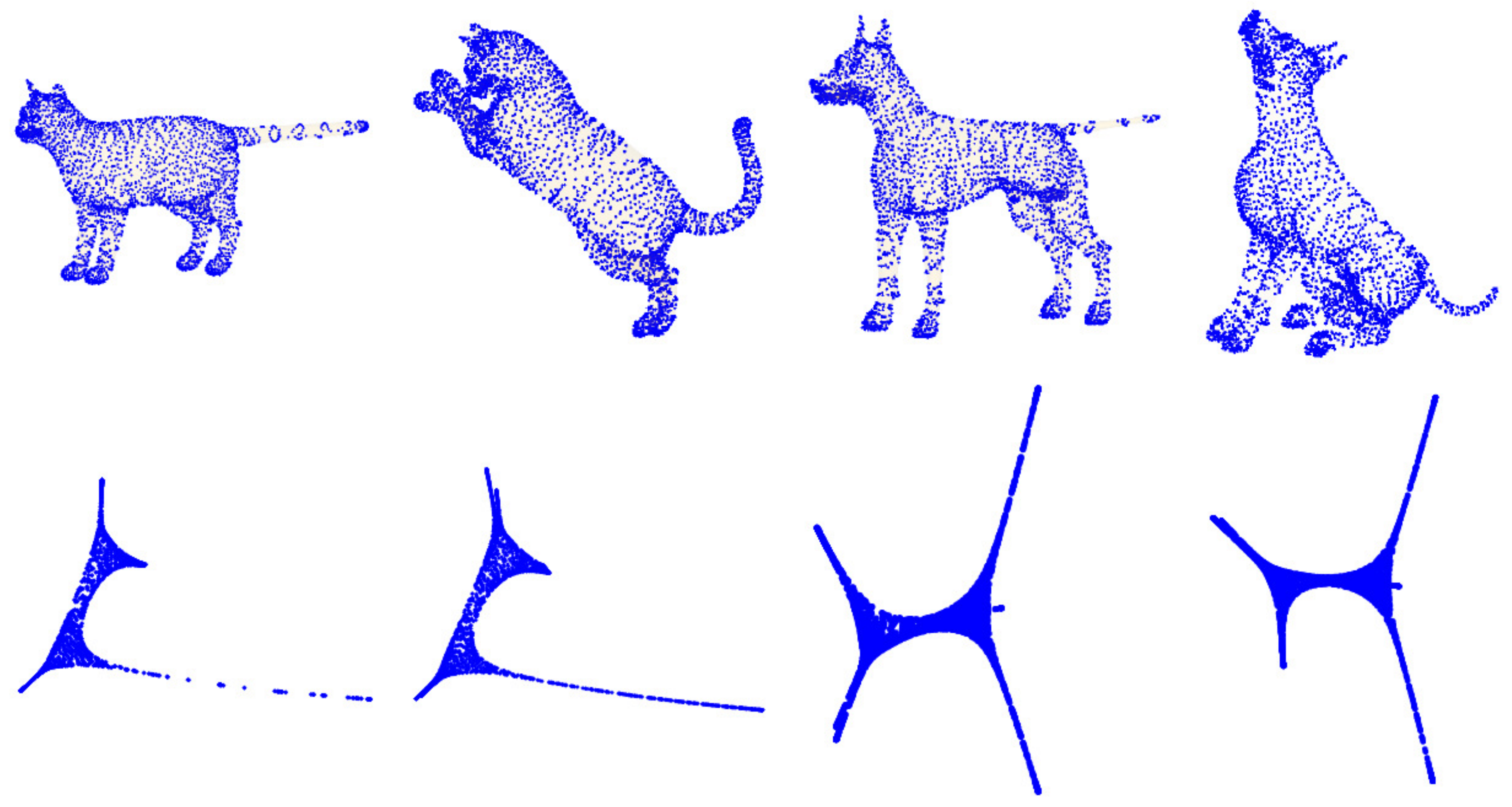}\\
\end{center}\vspace{-0.7cm}
\caption{LB eigenmap using the first three nontrivial LB eigenfunctions. }
\label{fig:LBeigenembedding}
\end{figure}

Note that the above LB eigenmap is scale invariant from Proposition~\ref{prop:scaleformula}. In addition, if a point cloud is sampled from a two dimensional manifold, this embedding is the same as the global point signature (GPS) proposed in~\cite{Rustamov:2007}. Figure. \ref{fig:LBeigenembedding} illustrates the LB eigenmap using the first three nontrivial LB eigenfunctions for point clouds sampled from cats and dogs, where each pair of point clouds for dogs and cats are sampled from surfaces which are close to isometric. First, we can clearly see the LB eigenmap is intrinsic while the coordinated representation of original point cloud embedded in Euclidean coordinates is not. So using LB eigenmap can 
remove isometric variance in non-rigid deformation automatically. Second, the LB eigenmap gives a natural multi-scale representation of the underlying manifold. Figure \ref{fig:LBeigenSmoothing} shows multi-scaled reconstructions using different number of LB eigenfunctions to represent $(x, y, z)$ coordinate functions of give data. It is clearly to see that detailed geometric information of point clouds can be gradually recovered using more and more LB eigenfunction. These two important properties are the motivation for our proposed registration model.

\begin{figure}[h]
\begin{center}
\includegraphics[width=1\linewidth]{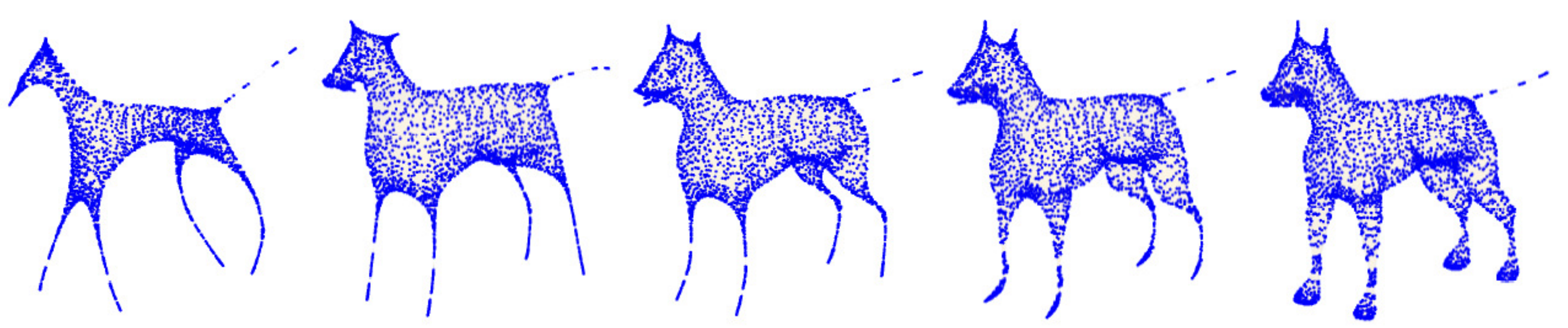}\\
\end{center}\vspace{-0.5cm}
\caption{Multi-scale representation using the first $n$ LB eigenfunctions to represent $(x, y, z)$ coordinate functions.}
\label{fig:LBeigenSmoothing}
\end{figure}

However, LB eigenmap introduces a few new problems due to ambiguities inherited in LB or any other eigen-system. For example, if $\phi_i$ is an LB eigenfunction, so does $-\phi_i$. Moreover, if the dimension $d_i$ of the eigen-space corresponding to $\lambda_i$, $\Big\{\phi~|~ \Delta_{\M}\phi = -\lambda_i\phi \Big\}$,  is greater than 1, then there is a freedom of orthonormal group $O(d_i)$ to choose those LB eigenfunctions. A rigorous distance has been discussed in~\cite{Berard:1994} to overcome this ambiguity while computation complexity is quite high. In addition, there is another potential issue coming from numerical computation of the LB eigen-system, where if two eigenvalues are too close, then ordering of LB spectrum may be switched. In Figure \ref{fig:LBeigenAmbiguity}, we illustrate examples for these ambiguities by showing the first four nontrivial LB eigenfunctions of two point clouds sampled from dog surfaces, which are obtained from the public available data base TOSCA~\cite{Bronstein:2006,Bronstein:2007,Bronstein:2008}. It is clearly that the pattern of the first nontrivial LB eigenfunctions are quite consistent for these two point clouds, while the second nontrivial LB eigenfunctions have a sign change and the third and fourth LB eigenfunctions have order switching for the two point clouds.

\begin{figure}[htp]
\begin{center}
\includegraphics[width=1\linewidth]{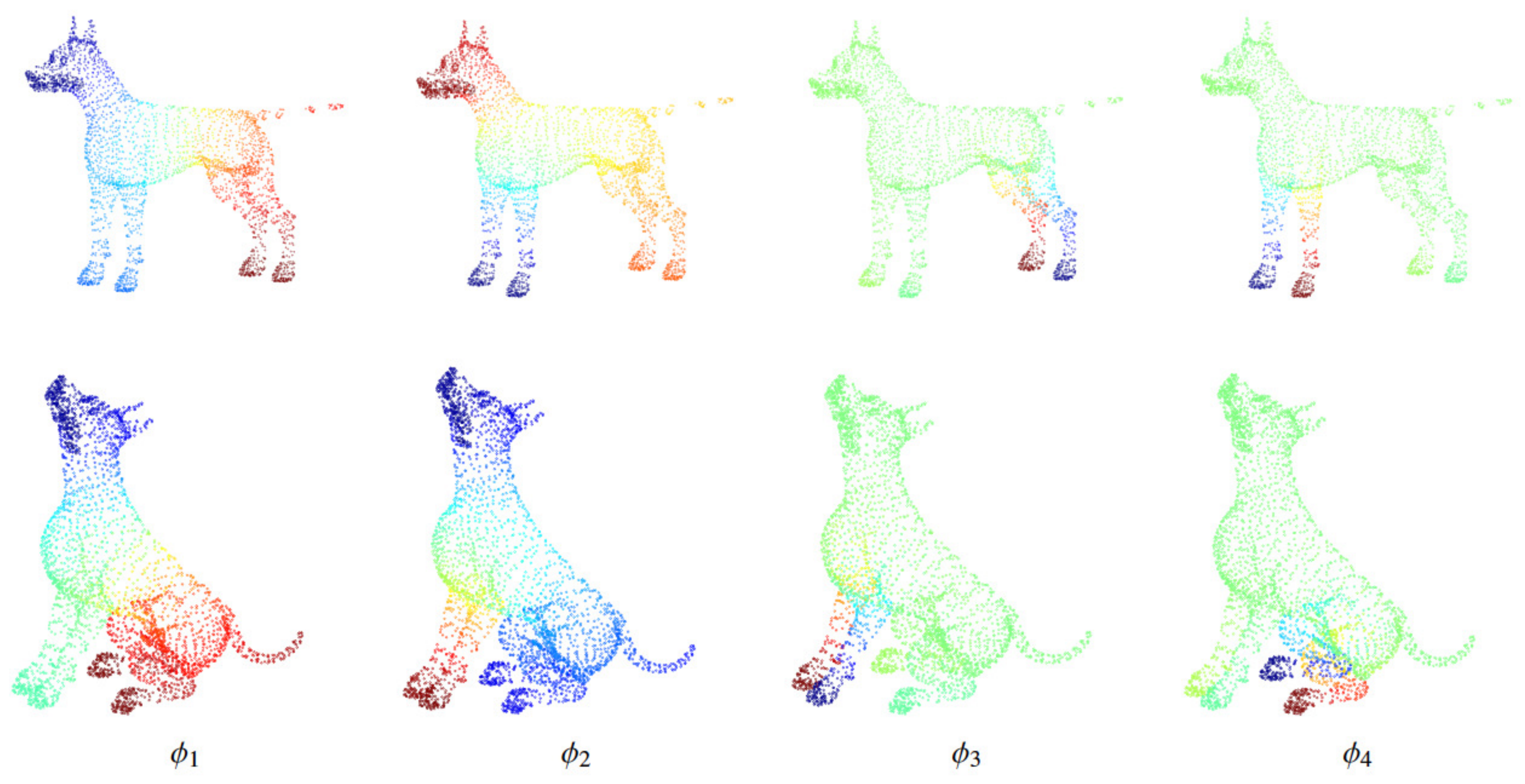}\\
\end{center}\vspace{-0.5cm}
\caption{Ambiguities of LB eigenmap, where the first four nontrivial LB eigenfunctions are color-coded on the point clouds. The color coding from red to blue means value of the corresponding LB eigenfunction varies from positive to negative.}
\label{fig:LBeigenAmbiguity}
\end{figure}

Although using LB eigenmap removes the variance of all isometric transformations in the original representations,  possible ambiguities described in above from the LB eigenmap can cause serious problems and lead to unsatisfactory point clouds registration. 
In fact, all these ambiguities can be modeled as an orthonormal group action on the image of LB eigenmap.  
More specifically, to handle ambiguities due to arbitrary sign, multiplicity or possible order switching of LB eigen-system, we would like to align $\P = \iota_{\M}^n (\M) \subset \RR^n$ and $ \Q = \iota_{\N}^n (\N) \subset \RR^n$ in $\mathbb{R}^n$ up to an orthonormal transformation.  Intuitively, a one-to-one and onto map should be constructed to have an alignment between $\P$ and $\Q$. In practice, however, certain discretization has to be introduced to sample $\M$ and $\N$. In many cases, sampling of $\M$ and $\N$ do not necessarily have the same number of points. Thus, a one-to-one and onto map does not make sense in this case. Here, we would like to view a registration in the distribution sense and propose the following registration model for $\P$ and $\Q$ based on the framework of optimal transport theory, which was first introduced by G. Monge~\cite{Monge:1781}  in 1781. A breakthrough of the optimal transport theory has been made in 1940's by Kantorovich~\cite{Kantorovich:1942} who proposed a relaxed version of Monge's problem allowing mass splitting, which is known as the Monge-Kantorovich problem described as follows.

Given two complete and separable metric spaces $\P, \Q \subset\RR^n$, we write $\sB(\P), \sB(\Q)$ as the sets of Borel measurable sets on $\P, \Q$ respectively, and write $\sP(\P), \sP(\Q)$ as the sets of Borel probability measures on $\P, \Q$ respectively. For fixed measures $\mu^{\P}\in\sP(\P), \mu^{\Q}\in\sP(\Q)$, we define the admissible set
\begin{equation}
\label{eq:ad}
\ADM(\mu^\P, \mu^\Q) = \Big\{\sigma \in\sP(\P\times\Q)~|~ \sigma(A\times \Q ) = \mu^\P(A), \forall A\in\sB(\P) \mbox{ and } \sigma(\P\times B) = \mu^\Q(B), \forall B\in\sB(\Q) \Big\}
\end{equation}
The Monge-Kantorovich optimal transportation problem is to find an optimal transport plan by minimizing:
\begin{equation}
 \min_{\sigma} \int_{\P\times\Q} \| p  - q\|^2_2 ~\d \sigma(p,q), \qquad \text{s.t. } \qquad  \sigma \in\ADM(\mu^\P, \mu^\Q)\label{eqn:MKOT}
\end{equation}
The existence of an optimal transport plan is guaranteed since the admissible set $\ADM(\mu^\P, \mu^\Q)$ is compact w.r.t. the weak topology in $\sP(\P\times\Q)$. The optimal $\sigma\in\ADM(\mu^\P, \mu^\Q)  $ for the above transportation problem can be viewed as a correspondence or mapping in distribution sense. The mass at each location $p\in \P$, which is $\mu(p)$, is transported to all $q\in\Q$ with distribution $\sigma(p,q)$ and vice versa. It also provides a probability interpretation of the correspondence:  $p$ maps to each $q$ with probability $\displaystyle\frac{\sigma(p,q)}{\mu(p)}$. Mathematically, the optimal transportation problem \eqref{eqn:MKOT} is a relaxed formulation of the Monge's original mass transportation problem~\cite{Monge:1781} where one wants to find a measurable map $T: \P\rightarrow \Q$ satisfying 
\begin{equation}
 \min_{T} \int_{\P} \| p  - T(p)\|^2_2 ~\d \mu^{\P}(p), \mbox{s.t.}\qquad T_{\#}\mu^{\P} = \mu^{\Q} 
\end{equation}
where $T_{\#}\mu^{\P}$ is the push forward of $\mu^{\P}$ by $T$. Monge's optimal transport problem can be ill-posed because admissible $T$ may not exist and the constraint $T_{\#}\mu^{\P} = \mu^{\Q}$ is not weakly sequentially closed. We use the Monge-Kantorovich optimal transportation~\eqref{eqn:MKOT} because it theoretically well-posed and provides more generality and flexibility for our problem as we will comment more later. 
The optimal transport theory has been widely used in many fields including economy~\cite{carlieroptimal}, fluid dynamics~\cite{benamou2000computational}, partial different equations~\cite{jordan1998variational}, optimization~\cite{bouchitte1997mathematiques} as well as image processing~\cite{Ni:2009local,Su:2013area,rabin2011wasserstein}. 
More details about the optimal transport theory can be found in \cite{Cedric2003topics, villani2008optimal, Ambrosio2013user}, from where we also adopted notations for this work. In our problem, we would like to align $\P$ and $\Q$ obtained from LB eigenmap up to an orthonormal matrix to remove those ambiguities inherited in LB eigenmap as discussed before. Therefore, we propose to model the registration problem as
\begin{equation}
\min_{R} \min_{\sigma} \int_{\P\times\Q} \|p R - q\|^2_2 ~\d \sigma(p,q), \qquad \text{s.t. } \qquad  \sigma \in\ADM(\mu^\P, \mu^\Q), \quad R\in O(n) = \{ R\in\mathbb{R}^{n\times n} ~|~ R R^T = I_n\}
\label{eqn:RWdistance}
\end{equation}

\begin{figure}[h]
\centering
\includegraphics[width=.7\linewidth]{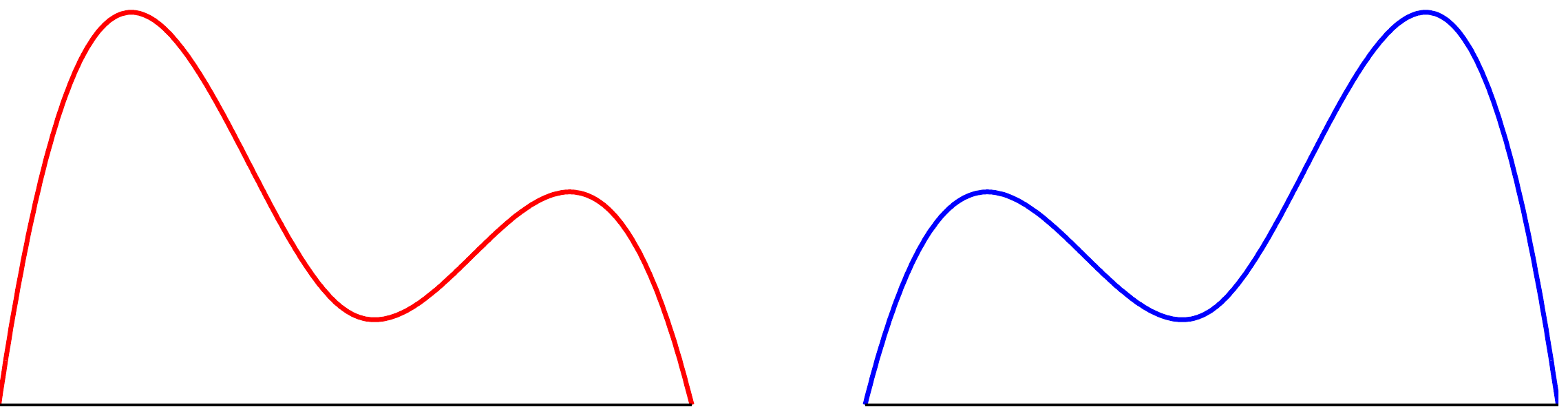}\\
\caption{Two distributions in $\RR$ up to a rotation $R\in O(1)$. }
\label{fig:1D_RWD}
\end{figure}
As an advantage of the above proposed model, the newly introduced rotation matrix $R$, can successfully handle possible ambiguity of the data representation. A simple example is illustrated in Figure~\ref{fig:1D_RWD}, where we consider $\P = \Q = \RR$ with two distributions up to a sign misassigment. In this case, the canonical Monge-Kantorovich model will not be able to figure out this misalignment, while this can be successfully tackled using~\eqref{eqn:RWdistance} as $O(1) = \{-1, 1\}$. Similar situation can also be handled in the case of high dimension data whose representation has ambiguity up to an orthonormal matrix. Therefore, the proposed model is more flexible and robust than the canonical Monge-Kantorovich model for data with possible misassignment up to $O(n)$. 
 
In practice, we assume $\M$ and $\N$ are sampled by two sets of points denoted by $ \{u_i\in\M \subset \mathbb{R}^D~|~ i = 1,\cdots,\ell_M\}$ with a mass distribution  $\mu^{\M}\in \RR_+^{\ell_M}$ and $ \{v_i\in \N\subset \mathbb{R}^D~|~ i = 1,\cdots,\ell_N\}$ with a mass distribution $\mu^{\N}\in \RR_+^{\ell_N}$. Hence, our task is to find a meaningful mapping or registration between two point clouds, $ \Big(\Big\{p_i = \iota_{\M}^n(u_i)\Big\}_{ i = 1}^{\ell_M},\mu^\M \Big)$ and $\Big( \Big\{q_i = \iota_{\N}^n(v_i)\Big \}_{i = 1}^{\ell_N}, \mu^\N \Big)$,  in $\mathbb{R}^n$ up to a orthonormal matrix using the optimal transportation model proposed in \eqref{eqn:RWdistance}. 
For convenience and consistency, we let $\ell_P = \ell_M, \ell_Q = \ell_N, \mu^P = \mu^\M, \mu^Q = \mu^\N$ and write row vectors $p_i$ and  $q_i$ as the $i-$th row of matrices $P \in\RR^{\ell_\P\times n}$ and $Q \in\RR^{\ell_\Q\times n}$ respectively. In other words, we need to find an optimal transport plan $\sigma \in\mathbb{R}^{\ell_P\times \ell_Q}$ and an orthonormal matrix $R$ such that the mass transportation from $(P R, \mu^P)$ to $(Q, \mu^Q)$ is minimal. In this discrete setting, the corresponding admissible set of optimal transportation plan is
\begin{equation}
\ADM(\mu^P,\mu^Q) = \Big\{\sigma = (\sigma_{ij})\in\mathbb{R}^{\ell_P\times \ell_Q} ~|~ \sigma \vec{1} = \mu^P, \sigma^T \vec{1} = \mu^Q, \sigma_{i,j} \geq 0, i = 1,\cdots,\ell_P,  j = 1,\cdots,\ell_Q  \Big\}
\label{eqn:discreteADM}
\end{equation}
where $\vec{1}$ denotes a column vector with all elements constant 1.
As a discrete version of the optimization problem \eqref{eqn:RWdistance}, the registration of two point clouds $P$ and $Q$ can be formulated as the following optimization problem:
\begin{equation}
\min_{R} \min_{\sigma} \sum_{i,j} \sigma_{i,j} ~\| p_i R  -q_j \|^2_2 ,\quad s.t. \quad R \in O(n),\quad \sigma \in \ADM(\mu^P,\mu^Q)
\label{eqn:Discrete_RWdistance}
\end{equation}
 Now we introduce a few definitions and notations to state properties of the proposed problem. We first define
\begin{equation}
\mathfrak{M}_n = \displaystyle \bigcup_{\ell=1}^{\infty}\left\{ (P,\mu) = \{(p_i,\mu_i)\in\RR^{n}\times \RR_{+} \}_{i=1}^\ell~|~ \displaystyle\sum_{i=1}^{\ell}\mu_i = 1, p_i\neq p_j \mbox{ if } i\neq j\right\}
\end{equation} as the space of all possible point clouds in $\RR^n$ that has $\ell=1, 2, \ldots$ number of points asscoiated with a measure $\mu$. We define an equivalence relation for two point clouds $(P,\mu^P),(Q,\mu^Q)\in \mathfrak{M}_n$, with number of points $\ell_P,\ell_Q$ respectively as 
$$(P, \mu^P) \thicksim (Q,\mu^Q) \Longleftrightarrow  
 \left\{\begin{array}{c} P \text{ and }  Q \text{ have the same number of points,  i.e. } \ell_P = \ell_Q \\
 \text{there exist $R\in O(n)$ and a permutation $\pi$ of } [1,\cdots,\ell_P] \quad \text{s.t.} \quad p_i R =  q_{\pi(i)},  ~~\mu^P_{i} = \mu^Q_{\pi(i)}
 \end{array}\right.
$$ 
We define robust Wasserstein distance (RWD) as
\begin{equation}
\label{eq:RWD}
\RWD\Big((P,\mu^P),(Q,\mu^Q)\Big)^2 = \min_{R \in O(n),\sigma\in \ADM(\mu^P,\mu^Q)} \sum_{i,j} \sigma_{i,j} ~\| p_i R  -q_j \|^2_2
\end{equation}

The following statements justify the well-posedness of the problem \eqref{eqn:Discrete_RWdistance} and the definition of $\RWD$.
\begin{theorem} 
\label{thm:Discrete_RW}
\begin{enumerate}
\item The optimization problem \eqref{eqn:Discrete_RWdistance} has at least one solution. 
\item  If $(P,\mu^P)\thicksim (P',\mu^{P'})$ and $(Q,\mu^Q)\thicksim (Q',\mu^{Q'})$, then $\RW\Big((P,\mu^P),(Q,\mu^Q)\Big) = \RW\Big((P',\mu^{P'}),(Q',\mu^{Q'})\Big) $.
\item $\RWD\Big((P,\mu^P),(Q,\mu^Q)\Big) \geq 0$, and $\RWD\Big((P,\mu^P),(Q,\mu^Q)\Big) = 0 \quad \Longleftrightarrow \quad (P,\mu^P) \thicksim (Q,\mu^Q)$.
\item $\RWD\Big((P,\mu^P),(Q,\mu^Q)\Big) =\RWD\Big((Q,\mu^Q),(P,\mu^P)\Big)$.
\item For any $(S,\mu^S) \in \mathfrak{M}_n$, $\RWD\Big((P,\mu^P),(Q,\mu^Q)\Big) \leq \RWD\Big((P,\mu^P),(S,\mu^S)\Big)+\RWD\Big((S,\mu^S),(Q,\mu^Q)\Big)$, 
\end{enumerate}
Thus, $\RWD(\cdot,\cdot)$ defines a distance on $\mathfrak{M}_n$.
\end{theorem}
\begin{proof}
1. Note that the discretized version of $\ADM(\mu^P,\mu^Q)$ defined in \eqref{eqn:discreteADM} is a compact set in $\mathbb{R}^{\ell_P\times \ell_Q}$, and $O(n)$ is also a compact set. Moreover, the cost function defined in \eqref{eqn:Discrete_RWdistance} is certainly smooth with respect to variables $\sigma$ and $R$.  Thus problem \eqref{eqn:Discrete_RWdistance} has at least one solution. However, the uniqueness of solution is not guaranteed due to the orthonormal matrix constraint. 

2.  If $(P,\mu^P)\thicksim (P',\mu^{P'})$, then there exists a permutation  $\pi$
and an orthonormal matrix $R_0\in O(n)$ such that $p_i R_0 =  p'_{\pi(i)},  ~~\mu^P_{i} = \mu^{P'}_{\pi(i)}, i = 1,\cdots,\ell_P$. Then 

\begin{eqnarray}
\displaystyle \RWD\Big((P',\mu^{P'}),(Q,\mu^Q)\Big)^2 &= &\min_{R \in O(n),\sigma\in \ADM(\mu^P\circ\pi^{-1},\mu^Q)} \sum_{i,j} \sigma_{i,j} ~\| p_{\pi^{-1}(i)} R_0R  -q_j \|^2_2   \nonumber \\
&\ndtstile{\hat{R} = R_0 R }{\hat{\sigma}_{i,j} = \sigma_{\pi(i),j} }& \min_{\hat{R} \in O(n),\hat{\sigma}\in \ADM(\mu^P,\mu^Q)} \sum_{i,j} \hat{\sigma}_{i,j} ~\| p_i \hat{R}  -q_j \|^2_2   \nonumber \\ 
& = & \RWD\Big((P,\mu^{P}),(Q,\mu^Q)\Big)^2 \nonumber
\end{eqnarray}
Similar, one can also show that $\RW\Big((P,\mu^{P}),(Q',\mu^{Q'})\Big)^2 = \RW\Big((P,\mu^{P}),(Q,\mu^Q)\Big)^2$, if $(Q,\mu^Q)\thicksim (Q',\mu^{Q'})$, which yields the first statement.

3. It is clear to see that $\RWD\Big((P,\mu^{P}),(Q,\mu^Q)\Big) \geq 0$ and $(P,\mu^P) \thicksim (Q,\mu^Q)$ implies $\RWD\Big((P,\mu^{P}),(Q,\mu^Q)\Big) = 0$. Next, we show $\RWD\Big((P,\mu^{P}),(Q,\mu^Q)\Big) = 0 \Longrightarrow (P,\mu^P) \thicksim (Q,\mu^Q)$. Let $(R^{PQ}, \sigma^{PQ})$ be an optimizer to attain $\RWD\Big((P,\mu^{P}),(Q,\mu^Q)\Big)$. Since $p_1,\cdots,p_{\ell_P}$ are distinct points, for any $q_j$, there is at most one $p_i R^{PQ}$ satisfying $p_i R^{PQ} = q_j$. Similar reason yields that for for any $p_i$, there is at most one $q_j$ satisfying $p_i R^{PQ} = q_j$. In addition, the condition that none of row vectors and column vectors of $\sigma^{PQ}$ are zeros vectors implies that $PR$ and $Q$ have one-to-one correspondence. In fact, $\Pi_{\sigma^{PQ}} = diag(\mu_{P_1}^{-1},\cdots,\mu_{P_{\ell_P}}^{-1}) \sigma^{PQ}$ is a permutation matrix satisfying $PR = \Pi_{\sigma^{PQ}} Q$ and $\mu^{P} = \Pi_{\sigma^{PQ}} \mu^Q$.

4. Using change of variables, we have:
\begin{eqnarray}
\displaystyle\RWD\Big((P,\mu^{P}),(Q,\mu^Q)\Big)^2 &=& \min_{R \in O(n),\sigma\in \ADM(\mu^P,\mu^Q)} \sum_{i,j} \sigma_{i,j} ~\| p_{i}  R  -q_j \|^2_2   \nonumber \\
&\ndtstile{\hat{R} = R^{-1}}{\hat{\sigma} = \sigma^T }& \min_{\hat{R} \in O(n),\sigma\in \ADM(\mu^Q,\mu^P)} \sum_{i,j} \hat{\sigma}_{i,j} ~\| q_{i}  \hat{R}  - p_j \|^2_2 \nonumber \\
& = & \RWD\Big((Q,\mu^Q),(P,\mu^{P})\Big)^2
\end{eqnarray}

5. Given $(P,\mu^P), (Q,\mu^{Q}), (S,\mu^{S}) \in\mathfrak{M}_n$, we denote $(R^{PS} ,\sigma^{PS})$ and $(R^{SQ} ,\sigma^{SQ})$ as optimizers of $\RWD\Big((P,\mu^{P}),(S,\mu^S)\Big)$ and $\RWD\Big((S,\mu^{S}),(Q,\mu^Q)\Big)$ respectively. In other words, we have
\begin{eqnarray}
\RWD\Big((P,\mu^{P}),(S,\mu^S)\Big) =  \left(\sum_{i,j} \sigma^{PS}_{i,j} ~\| p_{i}  R^{PS}  -s_j \|^2_2\right)^{1/2}, & \qquad \sum_{j=1}^{\ell_S}\sigma^{PS}_{i,j} = \mu^{P}_i, \qquad  \sum_{i=1}^{\ell_P}\sigma^{PS}_{i,j} = \mu^{S}_j  \\
\RWD\Big((S,\mu^{S}),(Q,\mu^Q)\Big) =  \left( \sum_{j,k} \sigma^{SQ}_{j,k} ~\| s_{j}  R^{SQ}  -q_k \|^2_2 \right)^{1/2}, &\qquad \sum_{k=1}^{\ell_Q}\sigma^{SQ}_{j,k} = \mu^{S}_j, \qquad  \sum_{j=1}^{\ell_S}\sigma^{SQ}_{j,k} = \mu^{Q}_k 
\end{eqnarray}

We define $\displaystyle \sigma^{PSQ}_{ijk} = \frac{1}{\mu^{S}_j}\sigma^{PS}_{ij}\sigma^{SQ}_{jk}$ and $R^{PSQ} = R^{PS}R^{SQ}$. It is clear that $\displaystyle\sum_j \sigma^{PSQ}_{ijk} \in\ADM(\mu^P,\mu^Q)$ as $\displaystyle \sum_{j,k} \sigma^{PSQ}_{ijk} = \mu^P_i$ and $\displaystyle \sum_{i,j} \sigma^{PSQ}_{ijk} = \mu^Q_k$. Moreover, we have

\begin{eqnarray}
\displaystyle \RWD\Big((P,\mu^{P}),(Q,\mu^Q)\Big)  & \leq&  \left(\sum_{i,k} \sum_{j} \sigma^{PSQ}_{i,j,k} ~\| p_{i}  R^{PSQ}  -  q_k \|^2_2 \right)^{1/2}   \nonumber \\
& \leq & \left(\sum_{i,j,k}  \sigma^{PSQ}_{i,j,k} ~\Big( \| p_{i}  R^{PS}R^{SQ}  - s_j R^{SQ} \|_2 + \| s_{j}  R^{SQ} - q_k  \|_2 \Big)^2 \right)^{1/2}   \nonumber \\
&\leq & \left( \sum_{i,j,k} \sigma^{PSQ}_{i,j,k} ~ \| p_{i}  R^{PS}R^{SQ}  - s_j R^{SQ} \|_2^2 \right)^{1/2}  + \left(\sum_{i,j,k} \sigma^{PSQ}_{i,j,k} \| s_{j}  R^{SQ} - q_k  \|_2 \Big)^2 \right)^{1/2}  \nonumber \\
& = & \left( \sum_{i,j} \sigma^{PS}_{i,j} ~ \| p_{i}  R^{PS} - s_j  \|_2^2 \right)^{1/2}  + \left(\sum_{j,k} \sigma^{SQ}_{j,k} \| s_{j}  R^{SQ} - q_k  \|_2 \Big)^2 \right)^{1/2}  \nonumber \\
& = & \RWD\Big((P,\mu^{P}),(S,\mu^S)\Big) + \RWD\Big((S,\mu^{S}),(Q,\mu^Q)\Big)
\end{eqnarray}
where the last inequality comes from the Cauchy-Schwartz inequality. This completes the proof.
\end{proof}

Let $\sigma^{PQ}$ be an optimizer that attains $\RWD((P,\mu^P),(Q,\mu^Q))$ defined in \eqref{eq:RWD}, we define $\Pi_{\sigma^{PQ}} = diag(\mu_{P_1}^{-1},\cdots,\mu_{P_{\ell_P}}^{-1}) \sigma^{PQ}$ as a map from $(P,\mu^P)$ to $(Q,\mu^Q)$ in the distribution sense. A map defined in distribution sense provides both generality and flexibility as well as gives a natural probability interpretation. For example, $\Pi_{\sigma^{PQ}}$ has a natural probability interpretation which says a point $p_i$ maps to $q_j$ with probability $\mu_i^{-1}\sigma_{i,j}$.  One can also think of 
$\mu_i^{-1}\sigma_{i,j}$ as weights to define an interpolative mapping $\displaystyle p_i\mapsto \mu_i^{-1}\sum_{j=1}^{l_Q}\sigma_{i,j} q_j$. Using this formulation one can deal with point clouds having different number of points, e.g., point clouds with different sampling rate, as well as incorporate uncertainty, local feature information, and prior knowledge into the measures $\mu$ and $\nu$. 
One can also use a post-process step to obtain correspondence in classical sense from $\sigma_{i,j}$. We show such examples in Section \ref{sec:NumericalResults}.

\begin{remark} 
\begin{enumerate}

\item $\RWD(\cdot,\cdot)$ can be naturally generalized to ``rotation" + ``translation", which can be used for more general registration problems. 
\item Since $\ADM(\mu^\P, \mu^\Q)$ defined in \eqref{eq:ad} is compact w.r.t. the weak topology in $\sP(\P\times\Q)$, one can similarly show that the proposed problem~\eqref{eqn:RWdistance} has at least one solution. Moreover, similar arguments used in this theorem can be used to show that the formula \eqref{eqn:RWdistance} also defines a distance for two spaces $\P, \Q \subset \RR^n$ with given mass distribution $\mu^{\P}, \mu^{\Q}$ respectively. 

\end{enumerate}
\end{remark}
The constraints of the non convex set $O(n)$ in \eqref{eqn:Discrete_RWdistance} make the above optimization problem difficult to solve. Here, we consider to solve the following two minimization problems alternatively to approach its solution. 
\begin{eqnarray}
 \displaystyle R^k &=& \arg\min_{R} \sum_{i,j} \sigma^{k-1}_{i,j} ~\| p_i R -q_j \|^2_2  ,\quad ~ \mbox{s.t.}  \quad R^TR= I_n  \\
\displaystyle \sigma^k& = &\arg\min_{\sigma} \sum_{i,j} \sigma_{i,j} ~\|  p_i R^k  -q_j \|^2_2 ,\quad ~ \mbox{s.t.}  \quad  \sigma \in \ADM(\mu^P,\mu^Q) 
\label{eqn:RW_sigma}
\end{eqnarray}

The first minimization problem has a closed-form solution as follows.  
\begin{eqnarray} \label{eqn:ONProj}
R^k &=& \arg\min_{R\in O(n)} \sum_{i,j} \sigma^{k-1}_{i,j} ~\| p_i R  -q_j \|^2_2  \nonumber \\
& =& \arg\min_{R\in O(n)} \sum_{i,j}\sigma^{k-1}_{i,j} p_i p_i^T + \sum_{i,j}\sigma^{k-1}_{i,j} q_j q_j^T -2  \sum_{i,j} \sigma^{k-1}_{i,j} p_i R q_j^T \nonumber \\
&=& \arg\min_{R\in O(n)}  -2 Tr(\sigma^{k-1} Q R^T P^T) \nonumber \\
&=& \arg\min_{R\in O(n)}  \|R - P^T\sigma^{k-1} Q\|^2_F \nonumber \\
& = & \mathrm{Proj}_{O(n)} (P^T\sigma^{k-1} Q) = UV^T
\end{eqnarray}
where $U, V$ are provided by SVD decomposition $P^T\sigma^{k-1} Q = UDV^T$. 
The second minimization problem can be solved by linear programming methods in $O(N^{2.5}\log(N))$ operations. 
Thus we have the following algorithm based on the proposed alternative method, the drawback of which is the computation cost mainly due to the second sub-optimization problem. In practice, large size of point clouds are usually expected, which makes this method quite time consuming. This is also our major motivation to introduce a new distance in the next section. 
\begin{algorithm2e}\caption{Iterative Method for \eqref{eqn:Discrete_RWdistance}}
\label{alg:RW}
Initialize $R^0, \sigma^0 $ \\
\While{``not converge"}{
$\displaystyle R^k = \arg\min_{R\in O(n)}  \sum_{i,j} \sigma^{k-1}_{i,j} ~\| p_i R -q_j \|^2_2 = UV^T$, where $\mbox{SVD}(P^T\sigma^{k-1} Q)   = UDV^T$.\\
Solve $\displaystyle \sigma^k = \arg\min_{\sigma} \sum_{i,j} \sigma_{i,j} ~\|  p_i R^k  -q_j \|^2_2 ,~ \mbox{s.t.} ~ \sigma \in \ADM(\mu_P,\mu_Q) $  by a linear programming method. }
\end{algorithm2e}

Due the non-convexity of the problem \eqref{eqn:Discrete_RWdistance}, it is hard to find its global minimizer. However, the following theorem suggests that the above alternative method will provide a monotone sequence to approach a local minimizer. 
\begin{theorem}For any $R\in O(n)$ and $\sigma \in \ADM(\mu_P,\mu_Q)$, if we write $\displaystyle E(R,\sigma) = \sum_{i,j} \sigma_{i,j} ~\|  p_i R  -q_j \|^2_2$, then the sequence $\{R^k,\sigma^k\}_{k=1}^\infty$generated by Algorithm \ref{alg:RW} satisfies:
\begin{equation}
E(R^{k+1},\sigma^{k+1}) \leq E(R^{k},\sigma^{k})
\end{equation}
\end{theorem}
\begin{proof} From the construction of $R^{k+1}$ and $\sigma^{k+1}$, we have:
\begin{equation}
E(R^{k+1},\sigma^{k}) \leq E(R,\sigma^{k}) \qquad \& \qquad  E(R^{k+1},\sigma^{k+1}) \leq E(R^{k+1},\sigma), \quad\mbox{for any} \quad R\in O(n), \sigma \in \ADM(\mu_P,\mu_Q) \nonumber
\end{equation}
Thus, let $R = R^k, \sigma = \sigma^k$ in the above two inequalities, we have:
\begin{equation}
E(R^{k+1},\sigma^{k+1}) \leq E(R^{k+1},\sigma^{k}) \leq E(R^{k},\sigma^{k}) \nonumber
\end{equation}
\end{proof}

\section{Robust Sliced-Wasserstein Distance and Numerical Algorithms}
\label{sec:RSWdistance}
To reduce the computation cost of the minimization problem discussed in the previous section, we first introduce a new distance called Robust Sliced-Wasserstein distance between two point clouds. The computation efficiency of the new distance is based on the existence of analytical solution of 1D optimal transport problem. We further develop an alternative method to solve the proposed nonconvex optimization problem. Moreover, we can also show that the proposed algorithm produces a decreasing sequence to approach a local minimizer. Inspired by the idea of the robust Sliced-Wasserstein distance, we also propose an empirical but more efficient algorithm for point clouds registration, which can be viewed as an approximation algorithm for model \eqref{eqn:Discrete_RWdistance}.

\subsection{1D optimal transportation}
\label{subsec:1DOT}
Let's first briefly discuss the solution of 1D optimal transport problem, which is our major motivation to propose a new distance introduced in section~\ref{subsec:RSWdistance}. Let $\sP(\RR)$ be the set of Borel probability measures on $\RR$. Given $\mu,\nu\in\sP(\RR)$, we write $F_{\mu}, F_{\nu}:\RR \rightarrow [0,~1]$ as the cumulative distribution functions of $\mu, \nu$ respectively. Define $F_{\mu}^{-1}:[0,~1]\rightarrow \RR$ as $F_{\mu}^{-1}(t) = \inf \{x\in\RR~|~ F_{\mu}(x) > t\}$, then the solution of 1D optimal transport problem is described by the following theorem \cite{Cedric2003topics}.
\begin{theorem} Consider the following 1D optimal transport problem
\begin{equation}
\sigma^{\mu,\nu} = \arg \min_{\sigma\in\ADM(\mu,\nu)} \int_{\RR\times\RR} (x - y)^2 \d \sigma(x,y) \nonumber
\end{equation}
then the following statements hold,
\begin{enumerate}
\item $\mbox{Supp}(\sigma^{\mu,\nu}) \subset \Big\{(x,y)\in\RR^2~|~ F_{\mu}(x^-) \leq F_{\nu}(y), \mbox{ and } F_{\nu}(y^-) \leq F_{\mu}(x) \Big\}$
\item For any $B\in\sB(\RR^2), \sigma^{\mu,\nu}(B) = \mathcal{L}\left(\Big\{t\in[0,~1]~|~(F_{\mu}^{-1}(t),F_{\nu}^{-1}(t)) \in B \Big\} \right)$, where $\mathcal{L}$ stands for the Lebesgue measure on $[0,~1]$.
\item The optimal value is given by $\displaystyle \int_0^1 (F_\mu^{-1}(s) - F_\nu^{-1}(s))^2 \d s =  \min_{\sigma\in\ADM(\mu,\nu)} \int_{\RR\times\RR} (x - y)^2 \d \sigma(x,y)  $
\end{enumerate}
\label{thm:1DOptimalTrans}
\end{theorem}

The above theorem suggests a closed-form solution for discrete 1D optimal transport problems. Given two discrete sets $X = \{x_1,  x_2, \cdots , x_{\ell_X}\} \subset \RR $ and $Y = \{y_1 ,  y_2 \cdots , y_{\ell_Y}\} \subset \RR $. We consider two discrete probability measures $\displaystyle \mu^X = \sum_{i=1}^{\ell_X} \mu^X_i \delta_{x_i} $ and $\displaystyle\mu^Y = \sum_{j=1}^{\ell_Y} \mu^Y_j \delta_{y_j}$. The discrete 1D optimal transport can be formulated as:
\begin{equation}
\sigma^{X,Y} = \arg \min_{\sigma\in\ADM(\mu^X,\mu^Y)} \sum_{i,j}\sigma_{i,j} (x_i - y_j)^2
\end{equation}
whose solution can be obtained from Theorem \ref{thm:1DOptimalTrans} with the following closed form.
\begin{corollary}[Solution of Discrete 1D optimal transport]
Let $\pi_y,\pi_y$ be two permutations such that $x_{\pi_x(1)} < x_{\pi_x(2)} < \cdots < x_{\pi_x({\ell_X})}$ and $y_{\pi_y(1)} < y_{\pi_y(2)} < \cdots < y_{\pi_y({\ell_Y})}$. Let $\hat{x}_i = x_{\pi_x(i)}, \hat{\mu}^X_i  = \mu^X_{\pi_x(i)}, s_0 = 0, s_i =  \hat{\mu}^X_{1} + \cdots + \hat{\mu}^X_{i}, i = 1,\cdots, {\ell_X} $ and $\hat{y}_j = y_{\pi_y(j)}, \hat{\mu}^Y_j  = \mu^Y_{\pi_y(j)}, h_0 = 0, h_j =  \hat{\mu}^Y_{1} + \cdots + \hat{\mu}^Y_{j}, j = 1\cdots, {\ell_Y}$.  Define 

\begin{equation}
\hat{\sigma}_{i,j}  =  \left\{\begin{array}{cc}
0, & \mbox{ if } s_i \leq h_{j-1}, \mbox{ or } ~~h_j \leq s_{i-1} \\
\hat{\mu}^X_i, & \mbox{ if }   h_{j-1} \leq s_{i-1} < s_i \leq h_j\\
\hat{\mu}^Y_j, & \mbox{ if }   s_{j-1} \leq h_{j-1} < h_j \leq s_i\\
s_i - h_{j-1}, & \mbox{ if }   s_{i-1} \leq h_{j-1} < s_i \leq h_j\\
h_j - s_{i-1}, & \mbox{ if }   h_{j-1} \leq s_{i-1} < h_j \leq s_i\\
\end{array}\right.
\end{equation}
Then $\sigma^{X,Y}_{i,j} = \hat{\sigma}_{\pi_x(i),\pi_y(j)}$.
\label{cor:1DDiscreteOT}
\end{corollary}

\begin{proof} 
Let $\displaystyle \hat{\sigma} = \arg \min_{\sigma\in\ADM(\hat{\mu}^x,\hat{\mu}^y)} \sum_{i,j}\sigma_{i,j} (\hat{x}_i - \hat{y}_j)^2$. According to Theorem \ref{thm:1DOptimalTrans}, we have
\begin{eqnarray*}
\hat{\sigma}_{i,j} &= &\mathcal{L}\left(\Big\{t\in[0,~1]~|~(F_{\hat{\mu}^X}^{-1}(t),F_{\hat{\mu}^Y}^{-1}(t)) = (\hat{x}_i,\hat{y}_j)  \Big\} \right) \\
& = &  \mathcal{L}\left(\Big\{t\in[0,~1]~|~ F_{\hat{\mu}^X}^{-1}(t) = \hat{x}_i \Big\} \bigcap \Big\{t\in [0,~1]~|~ F_{\hat{\mu}^Y}^{-1}(t)) = \hat{y}_j  \Big\} \right) \\
& = & \mathcal{L} \Big( (s_{i-1},~ s_{i}] \cap (h_{j-1},~ h_{j}] \Big) \\
& = & \left\{\begin{array}{cc}
0, & \mbox{ if } s_i \leq h_{j-1}, \mbox{ or } ~~h_j \leq s_i \\
\hat{\mu}^x_i, & \mbox{ if }   h_{j-1} \leq s_{i-1} < s_i \leq h_j\\
\hat{\mu}^y_j, & \mbox{ if }   s_{j-1} \leq h_{j-1} < h_j \leq s_i\\
s_i - h_{j-1}, & \mbox{ if }   s_{i-1} \leq h_{j-1} < s_i \leq h_j\\
h_j - s_{i-1}, & \mbox{ if }   h_{j-1} \leq s_{i-1} < h_j \leq s_i\\
\end{array}\right.
\end{eqnarray*}
\end{proof}
Particularly, if we set $\ell_X = \ell_Y = N$ and $\mu^X_i = \mu^Y_i = 1/N$, then $\hat{\sigma}_{i,j} = \frac{1}{N} \delta_{i,j}$, which is a well-known result for 1D discrete optimal transport between two sets with the same number of points and uniform mass distribution. As the result indicated from Corollary \ref{cor:1DDiscreteOT}, the computation complexity of the 1D discrete optimal transport problem relies on the cost of sorting algorithms, which can be as fast as $O(N\log N)$. The efficient algorithm for 1D optimal transport problem inspires us to propose the following {\em robust sliced-Wasserstein distance} (RSWD), which can be viewed as a generalization of {\em slice-Wasserstein distance} proposed in~\cite{Rabin2012wasserstein}, where distance is only defined for two sets with the same number of points and no rigid transformation flexibility is included.

\subsection{Robust Sliced-Wasserstein distance}
\label{subsec:RSWdistance}
Let $\P, \Q\subset \RR^n$ be two complete and separable metrizable topological spaces. Given a unit vector  $\theta\in S^{n-1} $, where $S^{n-1}$ denotes the unit sphere in $\RR^n$,  and an orthonormal matrix $R\in O(n)$, we define projection maps 
\begin{eqnarray}
\label{eq:proj}
\begin{array}{cc}
\pi^{\theta,R} : \P \rightarrow \RR   \\
\hspace{1.2cm} p \mapsto p R \theta^T
\end{array} \qquad \& \qquad
\begin{array}{cc}
\pi^{\theta} : \Q \rightarrow \RR   \\
\hspace{1cm} q \mapsto q \theta^T
\end{array}
\end{eqnarray}
For a given $\mu^\P\in\sP(\P)$, an induced measure $\pi^{\theta,R}_{\#}\mu^\P$ on $\RR$ is defined by $\pi^{\theta,R}_{\#}\mu^\P (B) = \mu^\P((\pi^{\theta,R})^{-1}(B)), \forall B\subset \sB(\RR)$. Similarly, we define induced measure $\pi^{\theta}_{\#}\mu^\Q$ on $\RR$ for a given measure $\mu^{\Q}$ on $Y$. In addition, we define 
\begin{equation}
\ADM(\pi^{\theta,R}_{\#}\mu^\P,\pi^{\theta}_{\#}\mu^\Q) = \{\sigma \in\sP(\RR\times\RR)~|~ \sigma(B\times\RR) = \pi^{\theta,R}_{\#}\mu^\P(B), ~\sigma(\RR\times B) = \pi^{\theta}_{\#}\mu^\Q(B), \forall B\in\sB(\RR) \}
\end{equation}
By adapting the Sliced-Wasserstein distance proposed in~\cite{Rabin2012wasserstein}, we define the following robust sliced-Wasserstein distance (RSWD). 
\begin{equation}
\label{eqn:RSW}
\RSWD\Big((\P,\mu^{\P}),(\Q,\mu^{\Q})\Big)^2 = \min_{R\in O(n)} \int_{S^{n-1}} \min_{\sigma \in\ADM(\pi^{\theta,R}_{\#}\mu^\P,\pi^{\theta}_{\#}\mu^\Q)}  \int_{\RR\times\RR} \|x - y \|_2^2 \d \sigma(x,y) ~\d \theta 
\end{equation}
For each fixed $R$, there exits an optimal transport plan  $\sigma^{\theta}\in\ADM(\pi^{\theta,R}_{\#}\mu^\P,\pi^{\theta}_{\#}\mu^\Q)$  for each direction $\theta \in S^{n-1}$. Since $O(n)$ is compact, there exists a minimizer $R$ that achieves 
$\RSWD\Big((\P,\mu^{\P}),(\Q,\mu^{\Q})\Big)$.
Similar to the discretized robust Wasserstein distance discussed in section \ref{sec:ManifoldRegistration},  let $\{p_1,\cdots,p_{\ell_P}\}$ and $\{q_1,\cdots,q_{\ell_Q}\}$ be the discrete sampling of $\P$ and $\Q$ respectively.  Again we write $P \in\RR^{\ell_\P\times n}$ and $Q \in\RR^{\ell_\Q\times n}$ for their matrices notation. Also denote $\mu_{\theta, R}^P$ and $\mu_{\theta}^Q$ to be the measures on $\RR$ induced by $\pi^{\theta, R}:~p_i \mapsto p_iR\theta^T$ and $\pi^{\theta} :~ q_j \mapsto q_j\theta^T$ respectively. Then the discrete version of RSWD  can be formulated as follows.
\begin{eqnarray}
\label{eqn:Discrete_RSW}
\RSWD\Big((P,\mu^{P}),(Q,\mu^Q)\Big)^2  & = & \min_{R\in O(n)} \int_{S^{n-1}} \min_{\sigma \in\ADM(\mu_{\theta, R}^P, \mu_{\theta}^Q)} \sum_{i,j}\sigma_{ij}\| p_i R \theta^T - q_j \theta^T \|^2 ~\d \theta  \nonumber
\\
& = &\min_{R\in O(n)} \int_{S^{n-1}} \sum_{i,j}\sigma_{ij}^{\theta,R}\| p_i R \theta^T - q_j \theta^T \|^2 ~\d \theta
\end{eqnarray}
where $\sigma^{\theta, R}\in \ADM(\mu_{\theta, R}^P, \mu_{\theta}^Q)$ denotes the optimizer for the optimal transport for the two 1D point sets $ \{p_iR\theta^T\}_{i=1}^{\ell_P}$ and $ \{q_j \theta^T\}_{j=1}^{\ell_Q}$ constructed as in Corollary \ref{cor:1DDiscreteOT} for each direction $\theta$.
Hence distance between two high dimensional objects are measured by averaging Wasserstein distances from all 1D projections. This distance provides us a useful tool for registration, comparison and classification of point clouds. 
Meanwhile, it has a great advantage in computation efficiency since each 1D problem can be solved by sorting in $O(Nlog(N))$ operations discussed in section \ref{subsec:1DOT}. In addition, the following theorem justifies the proposed RSWD as a distance.

\begin{theorem}The following statements hold for $\RSWD(\cdot,\cdot)$.
\begin{enumerate}
\item  If $(P,\mu^P)\thicksim (P',\mu^{P'})$, then $\RSWD\Big((P,\mu^P),(Q,\mu^Q)\Big) = \RSWD\Big((P',\mu^{P'}),(Q,\mu^{Q})\Big) $.
\item $\RSWD\Big((P,\mu^P),(Q,\mu^Q)\Big) \geq 0$, and $\RSWD\Big((P,\mu^P),(Q,\mu^Q)\Big) = 0 \quad \Longleftrightarrow \quad (P,\mu^P) \thicksim (Q,\mu^Q)$.
\item $\RSWD\Big((P,\mu^P),(Q,\mu^Q)\Big) =\RSWD\Big((Q,\mu^Q),(P,\mu^P)\Big)$.
\item For any $(S,\mu^S) \in \mathfrak{M}_n$, $\RSWD\Big((P,\mu^P),(Q,\mu^Q)\Big) \leq \RSWD\Big((P,\mu^P),(S,\mu^S)\Big)+\RSWD\Big((S,\mu^S),(Q,\mu^Q)\Big)$, 
\end{enumerate}
Therefore, $\RSWD(\cdot,\cdot)$ defines a distance on the space $ \mathfrak{M}_{n}/\thicksim$. 
\label{thm:Discrete_RSW}
\end{theorem}
\begin{proof}

1.  If $(P,\mu^P)\thicksim (P',\mu^{P'})$, then there exists a permutation  $\pi$
and an orthonormal matrix $R_0\in O(n)$ such that $p_i R_0 =  p'_{\pi(i)},  ~~\mu^P_{i} = \mu^{P'}_{\pi(i)}, i = 1,\cdots,\ell_P$. Then 
\begin{eqnarray}
\displaystyle \RSWD\Big((P',\mu^{P'}),(Q,\mu^Q)\Big)^2 &= & \min_{R\in O(n)} \int_{S^{n-1}} \min_{\sigma\in \ADM(\mu_{\theta, R}^P, \mu_{\theta}^Q)} \sum_{i,j} \sigma_{i,j} ~\| p_{\pi^{-1}(i)} R_0R \theta^T  -q_j \theta^T\|^2_2 ~\d \theta  \nonumber \\
&\ndtstile{\hat{R} = R_0 R }{\hat{\sigma}_{i,j} = \sigma_{\pi(i),j} }&  \min_{\hat{R}\in O(n)} \int_{S^{n-1}}  \min_{\hat{\sigma}\in \ADM(\mu_{\theta, R}^P, \mu_{\theta}^Q)} \sum_{i,j} \hat{\sigma}_{i,j} ~\| p_i \hat{R} \theta^T -q_j \theta^T\|^2_2~ \d \theta   \nonumber \\ 
& = & \RSWD\Big((P,\mu^{P}),(Q,\mu^Q)\Big)^2 \nonumber
\end{eqnarray}

2. It is clear to see that $\RSWD(P,Q) \geq 0$ and $P\thicksim Q$ implies $\RSWD(P,Q) = 0$. Next, we show $\RSWD(P,Q) = 0 \Longrightarrow (P, \mu^P) \thicksim (Q,\mu^Q)$. Let $R^{PQ}\in O(n)$ be a rigid transformation that attains $\RSWD\Big((P,\mu^{P}),(Q,\mu^Q)\Big)=0$. It is clear that $\{p_i R^{PQ}\}_{i=1}^{\ell_P} =   \{ q_j\}_{j=1}^{\ell_Q} $ implies existence of a permutation matrix $\Pi$ satisfying $R^{PQ} P = \Pi Q$, which also yields $\mu_{P} = \Pi \mu_{Q}$. Thus, it suffices to show $\{p_i R^{PQ}\}_{i=1}^{\ell_P} =   \{ q_j\}_{j=1}^{\ell_Q} $ which can be proved by contradiction as follows.  

Let's assume $\{p_i R^{PQ}\}_{i=1}^{\ell_P} \neq  \{ q_j\}_{j=1}^{\ell_Q}$. 
Without loss of generality, we assume $q_{j_0} \notin \{ p_iR^{PQ}\}_{i=1}^{\ell_P}$. Thus, there is a direction $\theta \in S^{n-1}$ such that $\| p_{i}  R^{PQ} \theta^T -q_{j_0} \theta^T \|^2_2 > 0 , \forall i=1, 2, \ldots, \ell_P$. By the continuity of $\| p_{i}  R^{PQ} \theta^T -q_{j_0} \theta^T \|^2_2$ with respect to $\theta$, there is an open set $\Omega \subset  S^{n-1}$ and has positive Lebesgue measure $\mu(\Omega)>c_1>0$ such that $\| p_{i}  R^{PQ} \theta^T -q_{j_0} \theta^T \|^2_2 > c_2>0, \forall \theta \in \Omega, i=1, 2, \ldots, \ell_P$.  Let $\sigma^{\theta}\in \ADM(\mu_{\theta, R}^P, \mu_{\theta}^Q)$ denote the optimizer of $\displaystyle \min_{\sigma \in\ADM(\mu_{\theta, R}^P, \mu_{\theta}^Q)} \sum_{i,j}\sigma_{ij}\| p_i R \theta^T - q_j \theta^T \|^2 $ for each direction $\theta$. Since $\displaystyle \sum_{i} \sigma_{i,j}^{\theta}=\mu_{\theta}^Q(q_j\theta^T)>0$, we have
\begin{eqnarray}
\RSWD\Big((P,\mu^P),(Q,\mu^Q)\Big) &= & \int_{S^{n-1}}\sum_{i,j} \sigma_{i,j}^{\theta} ~\| p_i R^{PQ}\theta^T  -q_j\theta^T \|^2_2~ \d \theta
\nonumber \\
& \ge & \int_{\Omega}   \sum_{i} \sigma_{i,j_0}^{\theta} ~\| p_i R^{PQ}\theta^T  -q_{j_0} \theta^T\|^2_2~ \d \theta
\nonumber \\
& > & c_1c_2\mu_{\theta}^Q(q_{j_0}\theta^T) > 0
\end{eqnarray}
This is in contradiction to $\RSWD\Big((P,\mu^P),(Q,\mu^Q)\Big) = 0$.

%
%

3. Using change of variables, we have:
\begin{eqnarray}
\displaystyle\RSWD\Big((P,\mu^{P}),(Q,\mu^Q)\Big)^2 &=& \min_{R\in O(n)} \int_{S^{n-1}} \min_{\sigma\in \ADM(\mu_{\theta, R}^P, \mu_{\theta}^Q)} \sum_{i,j} \sigma_{i,j} ~\| p_{i}  R \theta^T -q_j \theta^T \|^2_2 ~\d\theta \nonumber \\
&\ndtstile{\hat{R} = R^{-1}}{\hat{\sigma} = \sigma^T, \hat{\theta}^T = \theta^T R }& \min_{R\in O(n)} \int_{S^{n-1}} \min_{\hat{\sigma}\in \ADM(\mu_{\hat{\theta}, \hat{R}}^Q, \mu_{\theta}^P)} \sum_{i,j} \hat{\sigma}_{i,j} ~\| q_{i}  \hat{R} \hat{\theta}^T  - p_j \hat{\theta}^T\|^2_2 ~\d \hat{\theta}  \nonumber \\
& = & \RSWD\Big((Q,\mu^Q),(P,\mu^{P})\Big)^2
\end{eqnarray}

4. Given $(P,\mu^P), (Q,\mu^{Q}), (S,\mu^{S}) \in\mathfrak{M}_n$, we denote $R^{PS}$ and $R^{SQ}$ as optimizers of $\RSWD\Big((P,\mu^{P}),(S,\mu^S)\Big)$ and $\RSWD\Big((S,\mu^{S}),(Q,\mu^Q)\Big)$ respectively. 
For each $\theta$, we construct $\sigma^{\theta}_{i,j,k}$ from the two optimal 1D transport plans $\sigma^{\theta}_{i,j}$ for $ \{p_iR^{PS}\theta^T\}_{i=1}^{\ell_P},  \{s_j \theta^T\}_{j=1}^{\ell_S}$ and $\sigma^{\theta}_{j,k}$ for $ \{s_jR^{SQ}\theta^T\}_{j=1}^{\ell_S},  \{q_k \theta^T\}_{k=1}^{\ell_Q}$  
the same way as used in the proof of theorem \ref{thm:Discrete_RW}. We have
\begin{eqnarray}
\displaystyle \RSWD\Big((P,\mu^{P}),(Q,\mu^Q)\Big)  & \leq&  \left( \int_{ S^{n-1}} \sum_{i,j,k} \sigma^{\theta}_{i,j,k} ~\| p_{i}  R^{PS} R^{SQ} \theta^T -  q_k \theta^T\|^2_2  ~ \d \theta \right)^{1/2} \nonumber \\
& \leq & \left( \int_{S^{n-1}} \sum_{i,j,k}  \sigma^{\theta}_{i,j,k} ~\Big( \| p_{i}  R^{PS}R^{SQ} \theta^T  - s_j R^{SQ} \theta^T \|_2 + \| s_{j}  R^{SQ}\theta^T - q_k \theta^T  \|_2 \Big)^2  ~ \d \theta\right)^{1/2}   \nonumber \\
&\leq & \left(\int_{S^{n-1}} \sum_{i,j,k} \sigma^{\theta}_{i,j,k} ~ \| p_{i}  R^{PS}R^{SQ} \theta^T - s_j R^{SQ} \theta^T \|_2^2\d \theta \right)^{1/2}  + \left( \int_{S^{n-1}}\sum_{i,j,k} \sigma^{\theta}_{i,j,k} \| s_{j}  R^{SQ} \theta^T - q_k \theta^T \|_2 \Big)^2 \d \theta\right)^{1/2}  \nonumber \\
& = & \left( \int_{S^{n-1}} \sum_{i,j} \sigma^{\theta}_{i,j} ~ \| p_{i}  R^{PS}  \theta^T- s_j \theta^T \|_2^2 \d \theta \right)^{1/2}  + \left( \int_{S^{n-1}}\sum_{j,k} \sigma^{\theta}_{j,k} \| s_{j}  R^{SQ} \theta^T- q_k \theta^T \|_2 \Big)^2  \d \theta\right)^{1/2}  \nonumber \\
& = & \RSWD\Big((P,\mu^{P}),(S,\mu^S)\Big) + \RSWD\Big((S,\mu^{S}),(Q,\mu^Q)\Big)
\end{eqnarray}
\end{proof}

\begin{remark}
\begin{enumerate}
\item $\RSWD(\cdot,\cdot)$ can be easily generalized to ``rotation" + ``translation", which can also be used for more general registration problems such as those considered in texture missing and image registration. 

\item In particular, if we freeze $R = I_n$ and set $\ell_P = \ell_Q = N, \mu^P = \mu^Q =1/N$, the $\RSWD$ distance reduces to the sliced-Wasserstein distance $\displaystyle \text{SWD}(P,Q)^2 = \int_{\theta\in\mathbb{R}^n, \|\theta\|=1} \min_{permutation ~\pi} \sum_{i=1}^{N} (  p_i\theta^T  -  q_{\pi(i)}\theta^T )^2 \d \theta $ proposed in~\cite{Rabin2012wasserstein}, where the authors only consider to measure two sets with the same number of points. The proposed $\RSWD$ distance can be viewed as a generalization of slice-Wasserstein distance. It also has the built in flexibility to handle ambiguities due to rigid transformation.

\end{enumerate}
\end{remark}

\subsection{Numerical algorithms for solving the Robust Sliced-Wasserstein distance}
To solve the non-convex optimization problem \eqref{eqn:Discrete_RSW} for the Robust Sliced-Wasserstein distance, we propose to alternatively and iteratively optimize with respect to $R$ and $\sigma$. The main advantage of this approximation is that $\sigma_{\theta}$ at each step can be efficiently constructed by essentially a sorting algorithm as shown in Corollary \ref{cor:1DDiscreteOT} for a given direction $\theta$, which is much more efficient than solving $\eqref{eqn:RW_sigma}$ by linear programming methods. Here is the proposed iterative algorithm for  \eqref{eqn:Discrete_RSW}.

\begin{algorithm2e}\caption{Iterative method for computing Robust Sliced-Wasserstein distance}
\label{alg:IteraRSW}

Initialize $R^0 , \sigma^0(\theta) $, \\
\While{``not converge"}{
$\displaystyle R^{k} = \min_{R\in O(n)} \int_{S^{n-1}} \sum_{i,j}\sigma^{k-1}_{i,j}(\theta)\Big( p_i R \theta^T - q_j \theta^T \Big)^2 \d \theta$.\\
For each  $\theta\in S^{n-1}$, solve $\displaystyle\sigma^{k}(\theta) =  \min_{\sigma \in\ADM(\mu_{\theta, R}^P, \mu_{\theta}^Q)} \sum_{i,j}\sigma_{ij}\Big( p_i R^{k} \theta^T - q_j \theta^T \Big)^2$ by Corollary \ref{cor:1DDiscreteOT}.\\
} 
\end{algorithm2e}

Moreover, the following theorem guarantees that the proposed algorithm creates a monotone sequence to approach a local minimizer of \eqref{eqn:Discrete_RSW}.

\begin{theorem}
Let $\displaystyle E(R)  = \int_{\theta\in\mathbb{R}^n, \|\theta\|=1} \min_{\sigma \in\ADM(\mu_{\theta, R}^P, \mu_{\theta}^Q)}\sum_{i,j}\sigma_{ij}\Big( p_i R \theta^T - q_j \theta^T \Big)^2 \d \theta$ and $\{ R^k\}_{k=1}^{\infty}$ be a sequence generated by algorithm~\ref{alg:IteraRSW}, then $E(R^{k+1}) \leq E(R^{k})$.
\end{theorem}
\begin{proof} 
From the construction of $R^{k+1}$, it is clear that 
\begin{equation}
\int_{S^{n-1}}  \sum_{i,j}\sigma^{k}_{ij}(\theta)\Big( p_i R^{k+1} \theta^T - q_j \theta^T \Big)^2 \d \theta \leq \int_{S^{n-1}}  \sum_{i,j}\sigma^{k}_{ij}(\theta)\Big( p_i R^{k} \theta^T - q_j \theta^T \Big)^2  \d \theta
\end{equation}
Since  $\displaystyle\sigma^{k}(\theta) = \arg \min_{\sigma \in\ADM(\mu_{\theta, R}^P, \mu_{\theta}^Q)}\sum_{i,j}\sigma_{ij}\Big( p_i R^k \theta^T - q_j \theta^T \Big)^2$, it is true that
\begin{equation}
\int_{S^{n-1}} \sum_{i,j}\sigma^{k}_{ij}(\theta)\Big( p_i R^{k} \theta^T - q_j \theta^T \Big)^2 \d \theta = \int_{S^{n-1}}  \min_{\sigma \in\ADM(\mu_{\theta, R}^P, \mu_{\theta}^Q)}\sum_{i,j}\sigma_{ij}\Big( p_i R^k \theta^T - q_j \theta^T \Big)^2 \d \theta = E(R^k)
\end{equation}
On the other hand, since $\displaystyle\min_{\sigma \in\ADM(\mu_{\theta, R}^P, \mu_{\theta}^Q)}\sum_{i,j}\sigma_{ij}\Big( p_i R^{k+1} \theta^T - q_j \theta^T \Big)^2 \leq
\sum_{i,j}\sigma^{k}_{ij}(\theta)\Big( p_i R^{k+1} \theta^T - q_j \theta^T \Big)^2$, we have
\begin{eqnarray}
E(R^{k+1}) &=& \int_{S^{n-1}} \min_{\sigma \in\ADM(\mu_{\theta, R}^P, \mu_{\theta}^Q)}\sum_{i,j}\sigma_{ij}\Big( p_i R^{k+1} \theta^T - q_j \theta^T \Big)^2 \d \theta \nonumber \\
&\leq& \int_{S^{n-1}} \sum_{i,j}\sigma^{k}_{ij}(\theta)\Big( p_i R^{k+1} \theta^T - q_j \theta^T \Big)^2 \d \theta  \nonumber \\
&\leq& \int_{S^{n-1}}  \sum_{i,j}\sigma^{k}_{ij}(\theta)\Big( p_i R^{k} \theta^T - q_j \theta^T \Big)^2 \d \theta \nonumber \\
&=& E(R^k)  \nonumber
\end{eqnarray}
\end{proof}

To efficiently approximate the integration and solve the problem \eqref{eqn:Discrete_RSW} with orthogonality constraint, we propose the following curvilinear search algorithm. We denote $\Theta = \{\theta_1,\theta_2,\cdots, \theta_L\} \subset S^{n-1}$ as $L$ randomly chosen unit directions in $\mathbb{R}^n$. For each $\theta_l\in\Theta$, it is easy to compute the optimizer of $\displaystyle\sigma(\theta_l) = \arg\min_{\sigma \in\ADM(\mu_{\theta, R}^P, \mu_{\theta}^Q)} \sum_{i,j}\sigma_{ij}\Big( p_i R  \theta^T_l - q_j \theta^T_l \Big)^2$ by the sorting method provided in Corollary \ref{cor:1DDiscreteOT}.  The crucial step is to solve the following optimization problem with orthogonality constraint.
\begin{equation}\label{eqn:RSW_R}
\displaystyle \arg\min_{R\in O(n)} E_{\Theta}(R) = \sum_{l=1}^L \sum_{i,j}\sigma_{ij}(\theta_l)\Big( p_i R  \theta^T_l - q_j \theta^T_l \Big)^2
\end{equation}
Based on Cayley transform, a curvilinear search method with
Barzilai-Borwein (BB) step size is introduced in~\cite{Wen:2013feasible}. Here, we adapt this method to solve the above problem.  For convenience, we use the following notations:
\begin{eqnarray}
H &=& \displaystyle\nabla_{R} E_{\Theta}(R)  = 2 \sum_{l=1}^L  \sum_{i,j}\sigma_{ij}(\theta_l) p_i^T ( p_i R   -  q_j)  \theta_l\theta_l^T   \nonumber \\
A &=& H R^T - R H^T  \nonumber \\
R &\leftarrow& Y(\tau) =  \left( I + \frac{\tau}{2} A\right)^{-1} \left( 1 - \frac{\tau}{2} A\right)R
\end{eqnarray}

To make the maximal use of the computed $\|H\|^2$ and $R^TH$ at each iteration $s$, we  determine a step size $\tau_s$ that makes significant descent while still guarantees the convergence of the overall iterations. We also apply nonmonotone curvilinear 
search with an initial step size determined by the Barzilai-Borwein formulas, which were developed originally for $\mathbb{R}^n$ in \cite{BarzilaiBorwein1988,ZhangHager2004}. At iteration $k$, the step size is computed as
 \begin{equation}
 \label{eq:tau}
 \tau_{s,1} = \frac{
 \Tr(D_{s-1}^TD_{s-1}^T)}{
 |\Tr(D_{s-1}^TW_{s-1} )| } \quad \mbox{ or } \quad  \tau_{s,2} = \frac{|
 \Tr(D_{s-1}^TW_{s-1}) |}{\Tr(
 W_{s-1}^T W_{s-1})}, 
 \end{equation}
 where $D_{s-1}:= R_s-R_{s-1}$ and $W_{s-1}=A_sR_s -
 A_{s-1}R_{s-1} $.
 The final value for $\tau_k$ is a fraction (up to 1, inclusive) of $\tau_{s,1}$ or $\tau_{s,2}$ determined by the nonmonotone search in
 Lines 6 and 7 of Algorithm \ref{alg:SCS}, which enforce a  trend of descent in the objective value but do not require strict descent at each iteration. 
More details and the proof of convergence of this algorithm can be found in~\cite{Wen:2013feasible}.
 
 Assembling the above parts, we arrive at Algorithm \ref{alg:SCS}, in which
$\epsilon$ is a stopping parameter, and $\rho$, $\delta$, and $\xi$ are
curvilinear search parameters, which can be set to typical values as $10^{-4}$,
$0.1$ and $0.85$, respectively.

\begin{algorithm2e}\caption{Curvilinear search algorithm for \eqref{eqn:Discrete_RSW}}
\label{alg:SCS}
\SetKwInOut{Input}{input}\SetKwInOut{Output}{output}
\SetKwComment{Comment}{}{}
\KwIn{$P, Q$}
\KwOut{$R, \bar{\sigma}$}

Initialize $R^0, \sigma_l^0$ and randomly choose $L$ directions $\{\theta_1,\cdots,\theta_L\} \subset S^{n-1}$

\While{``not converge"}{

 $s \gets 0$.\\
\While{$\|\nabla E_{\Theta}(R^{k+1,s})\|>\epsilon$}{
Compute $\tau_s \gets \tau_{s,1} \delta^h$ or $\tau_s \gets \tau_{s,2} \delta^h$, where
$h$ is the smallest nonnegative integer satisfying $E_{\Theta}(Y_s(\tau_s)) \leq C_s +
\rho\tau_s E_{\Theta}'(Y_s(0)) $.\\
  $R^{k+1,s+1} \leftarrow Y_s(\tau_s).$

  $Q_{s+1} \leftarrow \xi Q_s +1$ and
  $\displaystyle C_{s+1} \leftarrow \frac{\xi Q_{s}C_{s} + E_{\Theta}(R^{k+1,s+1})}{Q_{s+1}}$.

  $s\gets s+1$.
         }
Solve each $\displaystyle \sigma^{k+1}(\theta_l) = \min_{\sigma} \sum_{ij} \sigma_{ij} \Big( p_i R^{k} \theta_l^T  -  q_j \theta_l^T  \Big)^2$ by the sorting method given in Corollary \ref{cor:1DDiscreteOT}. \\
} 
$\displaystyle R = R^k, \bar{\sigma} = \frac{1}{L}\sum_{l=1}^L \sigma^{k}(\theta_l)$ \\
\end{algorithm2e}

\begin{remark}
\begin{enumerate}
\item The major difference between the model of minimizing RSWD \eqref{eqn:RSW} and the general non-convex model \eqref{eqn:Discrete_RWdistance} is that there is no explicit permutation involved in the RSWD. The correspondence or registration results from a post processing in the form of a distribution. In other words, 
suppose the above algorithm stops at $R^k\in O(n)$ and $\{\sigma^k(\theta_l)\}_{l=1}^L$, the average transport plan $\displaystyle\bar{\sigma} = \frac{1}{L} \sum_{l=1}^L \sigma^k(\theta_l) $ satisfies 
$\bar{\sigma} \vec{1}= \mu^P,  \bar{\sigma}^T\vec{1} = \mu^Q$ and $ \bar{\sigma} \geq = 0$, which indicates that $\bar{\sigma} \in\ADM(\mu^P,\mu^Q)$. We define $\Pi_{\bar{\sigma}}  = diag(\mu_{P_1}^{-1},\cdots,\mu_{P_{\ell_P}}^{-1}) \bar{\sigma}$ as a correspondence between $P$ and $Q$ in the distribution sense. 
\item As another advantage of the RSWD, the computation of $\sigma(\theta)$ along each projection is independent, which can be further sped up by using parallel computating. Thus, the algorithm has great potential to handle point clouds of large sizes. However, the realization of parallel computing is beyond the topic of this paper. We will discuss more along this direction in our future work. 
\end{enumerate}
\end{remark}

\subsection{An empirical point clouds registration algorithm}
Inspired by the explanation of the correspondence in the distribution sense, we propose the following empirical algorithm to approximate the 
general non-convex model problem \eqref{eqn:Discrete_RWdistance} proposed in Section \ref{sec:ManifoldRegistration}. 
We approximate $\sigma^k$ in the second step of algorithm~\ref{alg:RW}
by averaging the optimal permutation matrices in each direction. 
\begin{algorithm2e}\caption{An empirical algorithm for \eqref{eqn:Discrete_RWdistance}}
\label{alg:IteraSW}
\SetKwInOut{Input}{input}\SetKwInOut{Output}{output}
\SetKwComment{Comment}{}{}
\KwIn{$P, Q$}
\KwOut{$R, \sigma$}

Initialize $R^0, \sigma^0  $, and randomly choose $L$ directions $\{\theta_1,\cdots,\theta_L\} \subset S^{n-1}.$ \\
\While{``not converge"}{
$\displaystyle R^k = \arg\min_{R\in O(n)}  \sum_{i,j} \sigma^{k-1}_{i,j} ~\| p_i R -q_j \|^2_2 = UV^T$, where $\mbox{SVD}(P^T\sigma^{k-1} Q)   = UDV^T$.\\
$ \displaystyle \sigma^k = \frac{1}{L} \sum_{l=1}^L \sigma_{\theta_l}$, where each $\displaystyle\sigma_{\theta_l} = \arg\min_{\sigma\in\ADM(\mu_{\theta_l, R}^P, \mu_{\theta_l}^Q)} \sum_{ij} \sigma_{ij} \Big( p_i R^k \theta^T_l  - q_j \theta^T_l  \Big)^2$ solved by the sorting method given in Corollary \ref{cor:1DDiscreteOT}
}
$R = R^k, \sigma = \sigma^k$.
\end{algorithm2e}

The above algorithm has several advantages listed as follows:
\begin{enumerate}
\item $R^k$ at each step can be directly solved by SVD decomposition, which is more efficient than iteratively solving problem \eqref{eqn:RSW_R} in algorithm~\ref{alg:SCS}. Thus, we can expect this empirical algorithm will be more efficient than the registration algorithm using RSWD.
\item  The approximation of $\sigma^k$ at each step make the model \eqref{eqn:Discrete_RWdistance} computationally tractable. Similarly to what we discussed about algorithm~\ref{alg:SCS}, the second step in algorithm~\ref{alg:IteraSW}, which is the main computation cost,  is easily parallelizable. This property has great benefit for handling registration of point clouds of large sizes. 
\end{enumerate}

\subsection{Multi-scale registration}
As a great advantage of $\displaystyle \iota_{\M}^n:\M\rightarrow \mathbb{R}^n,   \iota_{\M}^n(u) = \left(\frac{\phi_1(u)}{\lambda_1^{d/4}},\cdots, \frac{\phi_n(u)}{\lambda_n^{d/4}}\right)$, LB eigenmap  $\iota_{\M}^n$ provides a natural multiscale characterization of the manifold $\M$ with global information. This observation leads to a simple and natural multi-scale registration algorithm which enjoys robustness, efficiency and accuracy. First, we choose a sequence of numbers of LB eigenfunctions to construct LB eigenmap for point clouds $\{u_i \in\M\}_{i=1}^{\ell_{\M}}$ and $ \{v_i \in\N\}_{i=1}^{\ell_{\N}}$. Efficient and robust registration on coarse scale can be obtained from LB eigenmap with the first few leading LB eigenvalues and corresponding eigenfunctions. Then registration at finer and finer scales can be achieved effectively by using LB eigenmap with more and more LB eigenvalues and corresponding eigenfunctions using available coarser scale registration as the initial guess. More precisely, let $3 = n_0 < \cdots <  n_j < \cdots <  n_J$ be a given sequence of numbers. We write $P^{n_j} = \Big\{ \iota_{\M}^{n_j}(u_i) \Big\}_{i=1}^{\ell_{\M}}$ and $Q^{n_j} = \Big\{ \iota_{\N}^{n_j}(v_i) \Big\}_{i=1}^{\ell_{\N}}$ as the LB eigenmap of these two point clouds induced from their leading $n_j$ LB eigenvalues and corresponding eigenfunctions. For convenience, we also write $(R, \sigma) = PCRegs((P,\mu^\M), (Q,\mu^\N))$ for an registration algorithm with input $P, Q$ and output $R, \sigma$. In particularly, we can choose $PCRegs(\cdot, \cdot)$ as either algorithm~\ref{alg:SCS} or algorithm~\ref{alg:IteraSW}. Then, we propose the following multi-scale registration algorithm.

\begin{algorithm2e}\caption{Multi-scale registration}
\label{alg:MultiScaledRegs}
Initialize $(R_0,\sigma_0) = PCRegs((P^{n_0},\mu^\M) , (Q^{n_0},\mu^\N) )$ \\
\For{j = 1:J}{
$\displaystyle (R_j, \sigma_j) = PCRegs((P^{n_j},\mu^\M) , (Q^{n_j} , \mu^\N))$ using $\sigma_{j-1}$ as an initial guess.\\
 }
 $\Pi^{PQ} = diag(\mu_{P_1}^{-1},\cdots,\mu_{P_{\ell_P}}^{-1}) \sigma_J$
\end{algorithm2e}

\section{Numerical results}
\label{sec:NumericalResults}
In this section, numerical tests are presented to illustrate the proposed methods for point clouds registration. First, we compare the computation efficiency of Algorithm~\ref{alg:SCS}, Algorithm~\ref{alg:IteraSW} and also indicate the capability of the proposed algorithm for fixing the ambiguities of LB eigen-system for point cloud registration. Second, we illustrate the effectiveness of the proposed methods for mapping point clouds with different density distribution. Finally, we demonstrate that the simple and natural multi-scale registration algorithm \ref{alg:MultiScaledRegs} can improve both efficiency and accuracy. All numerical experiments are implemented by MATLAB in a PC with a 16G RAM and a 2.7 GHz CPU.

\begin{figure}[h]
\begin{center}
\includegraphics[width=1\linewidth]{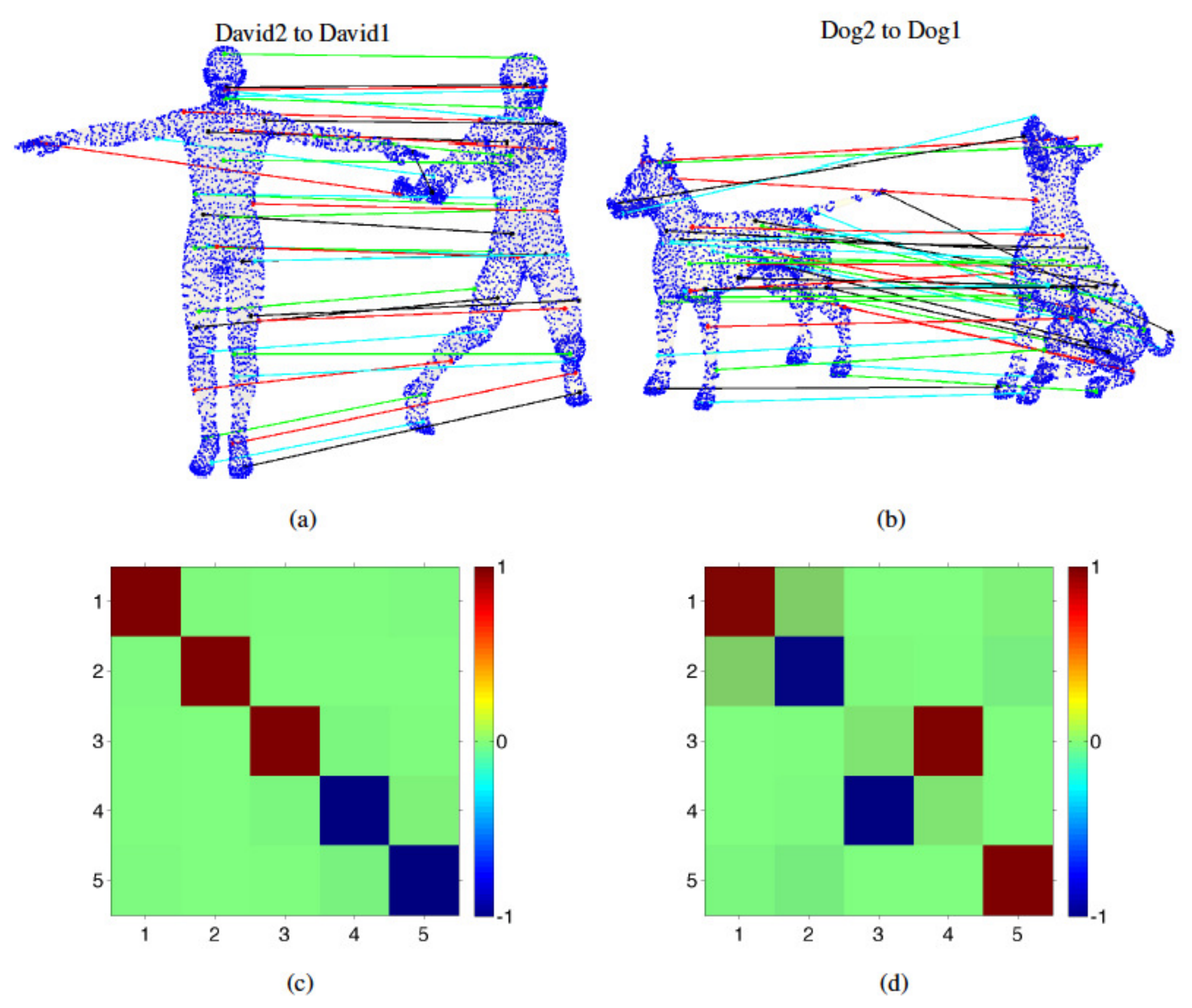}\\
\end{center}
\vspace{-0.5cm}
\caption{The first row: registration using the first 5 eigenvalues and corresponding eigenfunctions  by algorithm~\ref{alg:SCS} with $\mu^{P} = \mu^{Q} = 1/N$. The second row: the resulting $5\times 5$ rotation matrices for point clouds registration using the first 5 nontrivial LB eigenfunctions. }
\label{fig:PCReg_DavidDog}
\end{figure}

\begin{table}[h]
\begin{center}
\begin{tabular}{|c||c||c|c|c|c|c|c|}
\hline
\multirow{3}{2cm}{Point clouds} & \multirow{3}{1.2cm}{Size (N)}&\multirow{3}{1.5cm}{methods } & \multicolumn{5}{|c|}{Time consumption (s)} \\
 \cline{3-8}
 &  &  & $n = 5$  & $n = 10$  &  $n = 20$  &  $n = 30$   & $n = 50$    \\
 \cline{3-8}
  &  &  & $L = 1000$  & $L = 1500$  &  $L = 2000$  &  $L = 3000$   & $L = 5000$    \\
 \hline
\multirow{2}{2cm}{David point clouds} & \multirow{2}{1.2cm}{3400} & Alg. ~\ref{alg:SCS} &   5.23   & 26.40  &   64.11 & 93.66 & 168.46  \\
 \cline{3-8}
&  & Alg. ~\ref{alg:IteraSW}   & 13.91   &  18.36  &  23.92  & 34.39    & 55.10 \\
 \hline
 \multirow{2}{2cm}{Dog point clouds} &   \multirow{2}{1.2cm}{3400} & Alg. ~\ref{alg:SCS} &   17.40   & 42.50  &   60.32 & 92.20 &  167.04  \\
 \cline{3-8}
& & Alg. ~\ref{alg:IteraSW}   & 13.89   & 18.18 &  22.85  &  32.25    & 53.89 \\
 \hline
\end{tabular}
\end{center}
\caption{Comparison of computation for algorithm~\ref{alg:SCS} and ~\ref{alg:IteraSW}.}
\label{tab:PCReg_TimeComp}
\end{table}

\begin{figure}[h]
\centering
\includegraphics[width=1\linewidth]{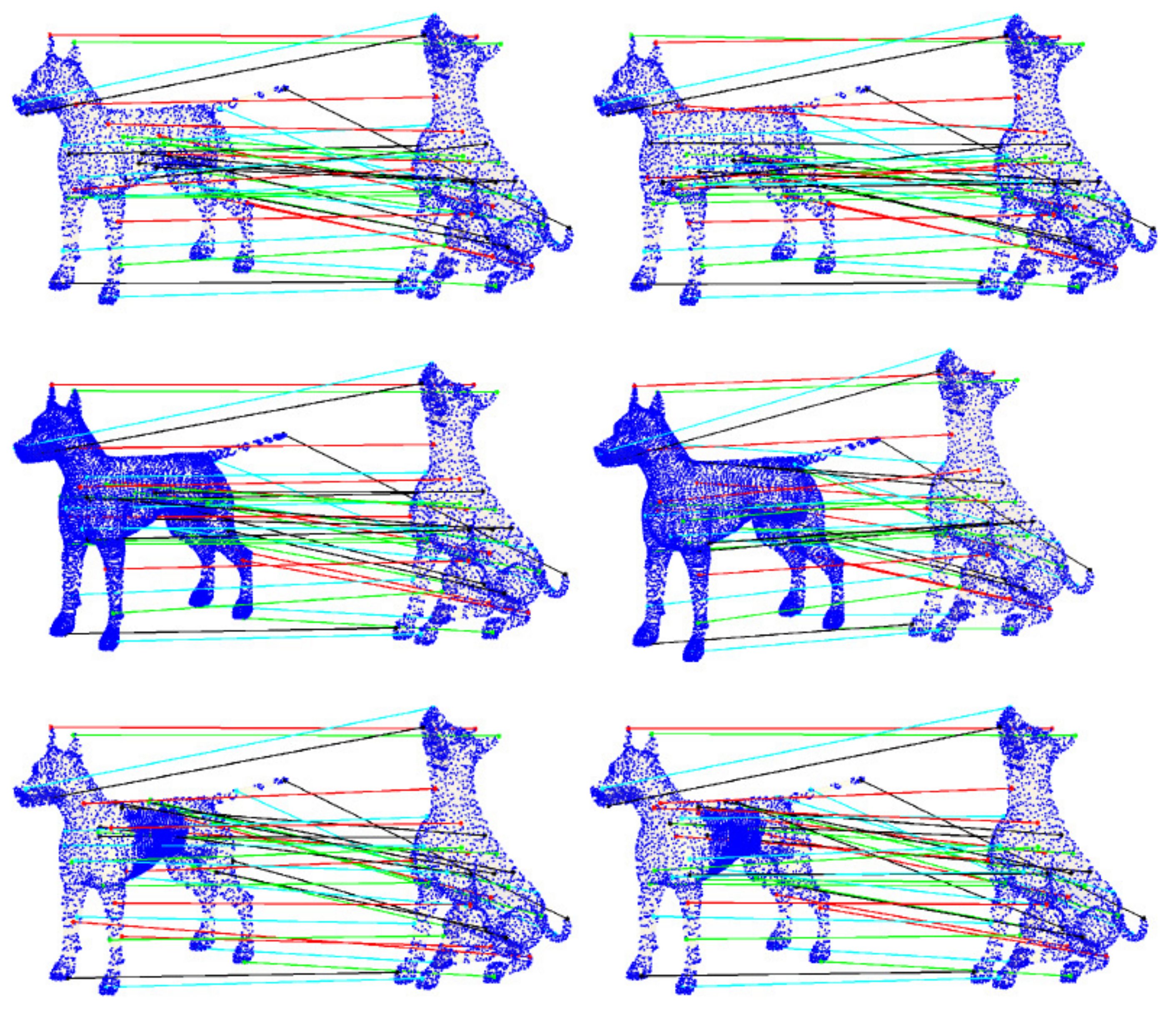}\\
\caption{Registration results between point clouds with different density distribution. Left column: registration results obtained by algorithm~\ref{alg:SCS} using the first 10 nontrivial LB eigenvalues and corresponding eigenfunctions. Right Column: registration results obtained by algorithm~\ref{alg:IteraSW} using the first 10 nontrivial LB eigenvalues and corresponding eigenfunctions.}
\label{fig:NonUniformDistributionPCRegistration}
\end{figure}

In our first example, we test the proposed algorithm~\ref{alg:SCS} and algorithm~\ref{alg:IteraSW} for registration of point clouds with the same number of points. Two point clouds sampled from two approximately isometric surfaces are given as input data. Note that the registration results obtained from both algorithms are correspondences in the distribution sense, which are represented as matrices whose row sum is 1.  We convert the correspondence matrix to pointwise correspondence only for visualization purpose by simply choosing the index of the maximal value for each row of the correspondence matrix. The first row of Figure \ref{fig:PCReg_DavidDog}  plots registration results using the first 5 nontrivial LB eigenvalues and corresponding eigenfunctions. The registration result is quite satisfactory visually. 
Meanwhile, as a by-product of the algorithms, we can obtain the optimal rigid transformation matrix for registering the two point clouds after LB eigenmap, which in fact shows the ambiguity of the LB eigenmap. As we illustrated in Figure~\ref{fig:LBeigenAmbiguity}, a rotation matrix for aligning LB embedding of two dog point clouds with the first 4 LB eigenfunctions should be close to 
$\left(\begin{array}{ccccc}
1\\
 & -1\\
 & & & 1\\
 & & -1 &
 \end{array}\right)
 $. In fact, this is consistent with our computation result illustrated in Figure~\ref{fig:PCReg_DavidDog} (d). Similarly, Figure~\ref{fig:PCReg_DavidDog} (c). reports the rotation matrix for two David point clouds registration. Moreover, we also list computation time of both algorithms in Table ~\ref{tab:PCReg_TimeComp}, which clearly shows the efficiency of our algorithms. 
 
 As an advantage of the proposed registration methods based on the optimal transport theory, our methods can also handle registration for data with different density distributions. In Figure~\ref{fig:NonUniformDistributionPCRegistration}, we report the second experiment  results for registration of point clouds with possible different density distribution using the proposed methods, where $\mu^{P},\mu^{Q}$ are chosen as the density distribution of the given data which can be approximated by computing the area of Voronoi cell of each point from its local mesh construction. Our results show that the proposed methods can handle registration problem for point clouds with different density distribution. 
 In summary, our computation results suggest that the proposed registration methods successfully address the ambiguity introduced from the LB eigenmap, provide efficient non-rigid point cloud registration and has great potential to handle non-rigid registration between two general point clouds with possible different density distributions. 

In our third experiment, we demonstrate the capability of the proposed multi-scale registration algorithm \ref{alg:MultiScaledRegs}. A nice property of our muli-scale algorithm is that it is not necessary to obtain perfect registration results at coarse scale, which can significantly speed up the whole registration process with even 200 LB eigenfunctions used as the finest level. In order to have a better quantitative assessment of the registration quality especially at the fine scale,  rather than using typical eyeball norm or visual inspection, we propose a {\em connectively transformation test} as follows. 

Let $\P = \{p_1,\cdots,p_N\}$ and $\Q = \{q_1,\cdots,q_N\}$ be two set of points sampled from the source surface $\M$ and target surface $\N$. Let $f$ be a map from $\P$ to $\Q$ such that $f(p_i) = q_{\sigma(i)}$. We impose a triangle mesh connectivity structure $\T = \Big\{ T_s = [p_{s_1}~ p_{s_2}~ p_{s_3}] \Big\}_s$ on the source point cloud $\P$ such that $\{\P, \T\}$ forms a triangle mesh representation of $\M$. Naturally, the imposed connectivity can be transformed to the target point cloud using the obtained registration result. In other words, we consider $f(\T) = \Big\{ f(T_s) = [f(p_{s_1})~ f(p_{s_2})~ f(p_{s_3})] \Big\}_s$ as an induced connectively on $\Q$ from $f$. It is certainly true that $\{\Q, f(\T)\}$ is not necessary a good triangle mesh representation of $\N$ for an arbitrary $f$. On the other hand, if $f$ provides a good correspondence or registration, then $\{\Q, f(\T)\}$ can provide a reasonable triangle mesh representation of $\N$. Namely, one can observe a regular triangle mesh representation of $\N$ without having cross-over triangles. The more accurate registration we have, the  smoother surface we can reconstruct using $\{\Q, f(\T)\}$. 

In Figure~\ref{fig:PC_MultiRegs_SCS_dog} and Figure~\ref{fig:PC_MultiRegs_IteraSW_dog}, we show the reconstructions from the transformed connectivity obtained from the intermediate registration results in the multi-scale registration process using algorithm~\ref{alg:SCS} and algorithm~\ref{alg:IteraSW} respectively. In our example, we choose a sequence of numbers $\{3, 5, 10, 20, 30, 50, 80, 120, 150, 200 \}$ for LB eigenembedding. We only run 2 iterations using the proposed registration algorithms at each level. It is clearly seen that the registration results at the first few levels only capture coarse level information. However, more and more detailed information can be captured with more and more LB eigenfunctions. Thus, we can see finer and finer structure from surface reconstructions. Although we only run 2 iterations without exactly solving the registration problem at each level, we observe that the multi-scale registration process can gradually correct the registration error and finally lead to very accurate registration results. In Tabel~\ref{tab:MultiScaleReg_TimeComp}, we report the computation time using algorithm~\ref{alg:SCS} and algorithm~\ref{alg:IteraSW} for each level respectively.

\begin{figure}[h]
\centering
\includegraphics[width=1\linewidth]{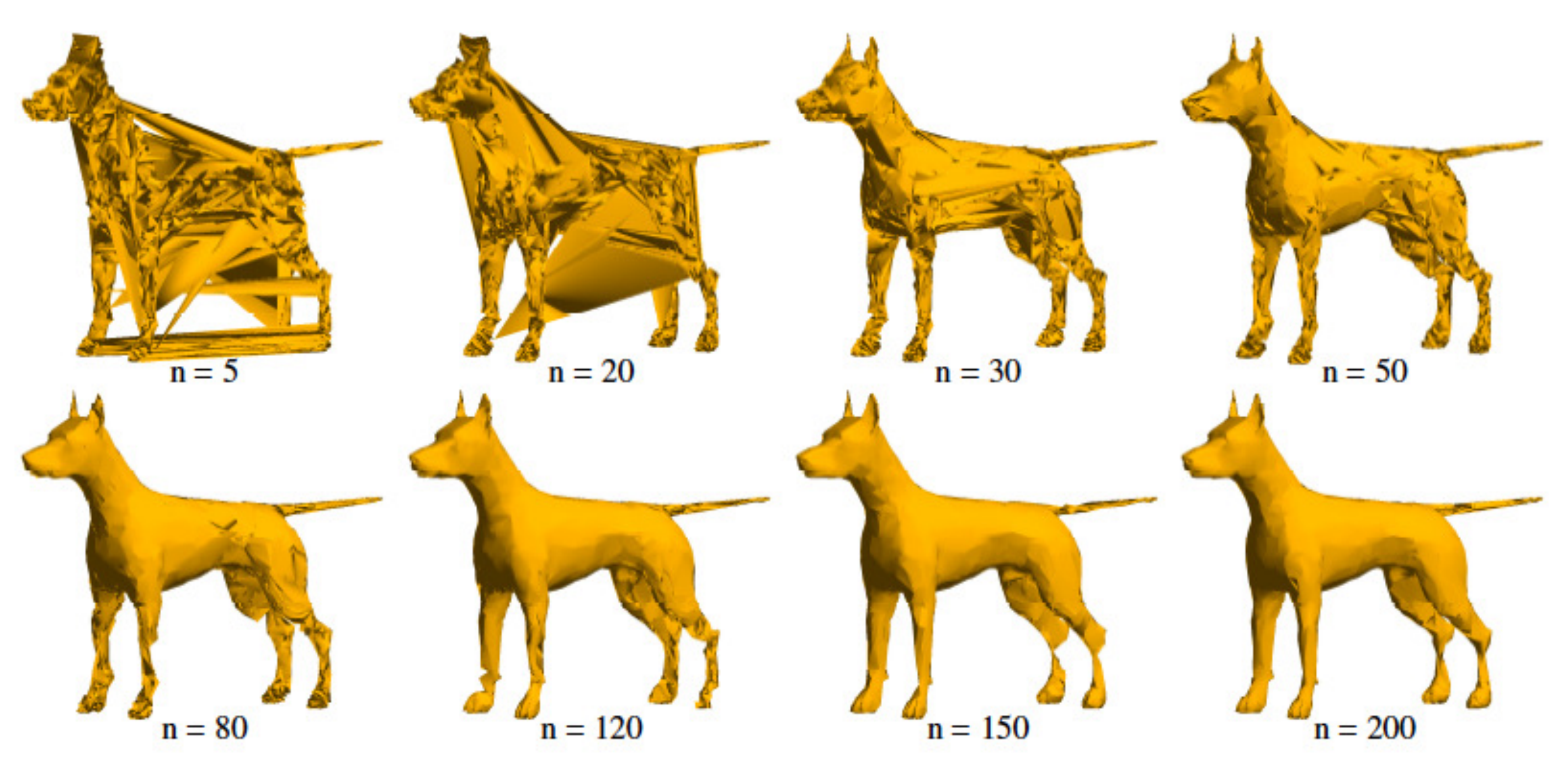}\\
\caption{Connectively transformation test of multi-scale registration using algorithm~\ref{alg:SCS}.}
\label{fig:PC_MultiRegs_SCS_dog}
\end{figure}

\begin{figure}[h]
\centering
\includegraphics[width=1\linewidth]{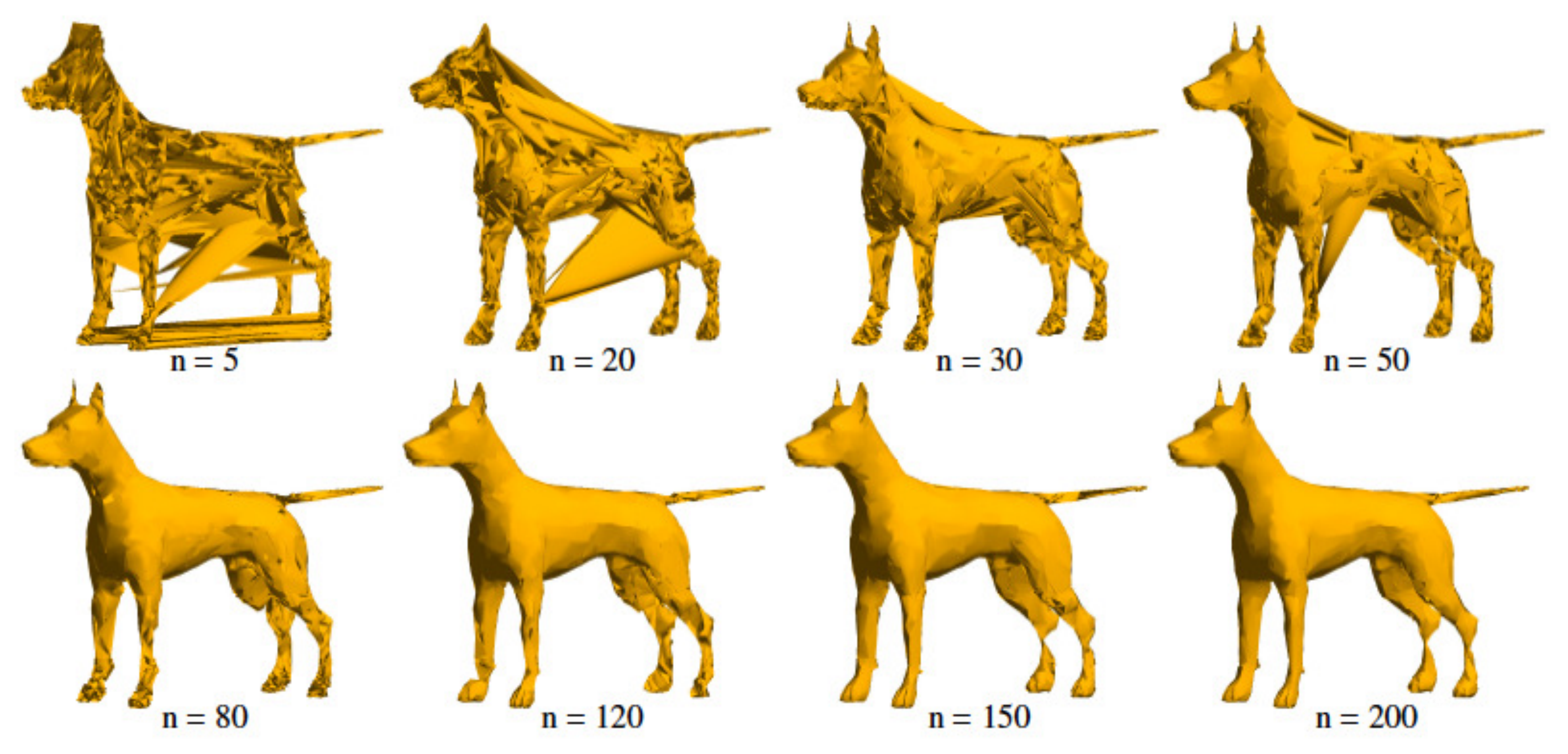}\\
\caption{Connectively transformation test of multi-scale registration using algorithm~\ref{alg:IteraSW}.}
\label{fig:PC_MultiRegs_IteraSW_dog}
\end{figure}

\begin{table}[h]
\begin{center}
\resizebox{6.5in}{.48in}{
\begin{tabular}{|c||c|c|c|c|c|c|c|c|c|c|}
\hline
 \cline{2-9} 
\multirow{3}{1.5cm}{methods } & \multicolumn{9}{|c|}{Time consumption (s)} \\
 \cline{2-10}
 &  $n_1 = 5$  & $n_2 = 10$  &  $n_3 = 20$  &  $n_4 = 30$   & $n_5 = 50$  & $n_6 = 80$  &   $n_7 = 120$  &   $n_8 = 150$ & $n_9 = 200$  \\
 \cline{2-10}
 & $L = 500$  & $L = 800$  &  $L = 1000$  &  $L = 1500$   & $L = 3000$   & $L = 6000$  &  $L = 10000$  &  $L = 15000$   & $L = 20000$  \\
 \hline
Alg. ~\ref{alg:SCS} &   0.61   & 1.21 &  1.50 & 2.26 &  4.75 &  9.58 & 16.50  &  26.23 & 38.42  \\
 \hline
Alg. ~\ref{alg:IteraSW}   & 0.45   &  .70  &  0.83  &  1.23   & 2.30 &  4.76 &  7.68  &  11.67  & 15.60 \\
 \hline
\end{tabular}
}
\end{center}
\caption{Comparison of computation time using algorithm~\ref{alg:SCS} and algorithm~\ref{alg:IteraSW} for multi-scale registration.}
\label{tab:MultiScaleReg_TimeComp}
\end{table}

\section{Conclusion}
\label{sec:conclusion}
In this work a multi-scale non-rigid point cloud registration model and computational algorithms are proposed. The key ideas in our approach are:
\begin{enumerate}
\item
Use LB eigenmap, which is invariant under isometric transformation and scaling, to give an intrinsic multi-scale representation of the original point clouds in the embedding space. 
\item
Develop effective registration models and computational algorithms for the new point clouds after LB eigenmap using non-convex optimization that can remove ambiguities introduced by LB eigenmap. In particular, robust sliced-Wasserstein distance (RSWD)  is proposed and proved to be a well defined metric and computationally efficient for point clouds.
\item
Employ a natural multi-scale procedure that can achieve efficiency, accuracy and robustness. 
\end{enumerate}
Numerical tests demonstrate promising results for our methods. There are a few interesting projects based on the current work we will pursue in the future. We will investigate possibilities to incorporate local or other features in the LB eigenmap and the measure, for instance affine invariant Gaussian curvature proposed in \cite{Raviv:2014Affine}. We will also study shape classification by using machine learning techniques based on LB eigenmap and RSWD. Last but not least, we will explore point cloud registration and classification in high dimensional data analysis.


\end{document}